\documentclass[11pt]{article}
\usepackage[a4paper,margin=1.0in]{geometry}

\usepackage[colorlinks,citecolor=red,linkcolor=red,backref=page]{hyperref}
\pdfpagewidth=\paperwidth \pdfpageheight=\paperheight
\usepackage{amsfonts,amssymb,amsthm,amsmath,eucal,tabu,url}
\usepackage{pgf}
 \usepackage{array}
 \usepackage{indentfirst}
\usepackage{tabularx} 

\usepackage{subfigure}

\theoremstyle{plain}
\newtheorem{thm}{Theorem}[section]
\newtheorem{theorem}[thm]{Theorem}

\newtheorem{proposition}[thm]{Proposition}

\newtheorem{conjecture}[thm]{Conjecture}

\theoremstyle{definition}
\newtheorem{definition}[thm]{Definition}
\newtheorem{remark}[thm]{Remark}

\newtheorem{thevarthm}[thm]{\varthmname}

\newenvironment{varthm*}[1]{\trivlist\item[]{\bf #1.}\it}{\endtrivlist}

\newcommand\newop[2]{\def#1{\mathop{\rm #2}\nolimits}}
\newop\log{log}
\newop\ord{ord}
\newop\Gal{Gal}
\newop\SL{SL}
\newop\Bl{Bl}
\newop\mult{mult}
\newop\mass{mass}
\newop\div{div}
\newop\codim{codim}
\newop\sing{sing}
\newop\vdim{vdim}
\newop\edim{edim}
\newop\Ass{Ass}
\newop\size{size}
\newop\reg{reg}
\newop\satdeg{satdeg}
\newop\supp{supp}
\newop\Neg{Neg}
\newop\Nef{Nef}
\newop\Nefh{Nef_H}
\newop\Eff{Eff}
\newop\Zar{Zar}
\newop\MB{MB}
\newop\MBxC{MB\mathit{(x,C)}}
\newop\NnB{NnB}
\newop\Bigg{Big}
\newop\Effbar{\overline{\Eff}}

\def\keywordname{{\bfseries Keywords}}%
\def\keywords#1{\par\addvspace\medskipamount{\rightskip=0pt plus1cm
\def\and{\ifhmode\unskip\nobreak\fi\ $\cdot$
}\noindent\keywordname\enspace\ignorespaces#1\par}}
\def\subclassname{{\bfseries Mathematics Subject Classification
(2020)}\enspace}
\def\subclass#1{\par\addvspace\medskipamount{\rightskip=0pt plus1cm
\def\and{\ifhmode\unskip\nobreak\fi\ $\cdot$
}\noindent\subclassname\ignorespaces#1\par}}

\begin{document}
\title{Klein's arrangements of lines and conics}
\author{G\'abor G\'evay and Piotr Pokora}
\maketitle
\thispagestyle{empty}
\begin{abstract}
In this paper we construct several arrangements of lines and/or conics that are derived from the geometry of the Klein arrangement of $21$ lines in the complex projective plane.
\subclass{52C35, 14N20, 32S22}
\end{abstract}

\section{Introduction}
In 1878, Felix Klein has found and studied in detail a very important curve with remarkable properties which bears his name~\cite{Kle}. This curve can be defined by the following homogeneous equation
$$\Phi_{4} \, : \, x^{3}y + y^{3}z + z^{3}x = 0,$$
and its very surprising properties come from the fact that this is an example of the so-called Hurwitz curves~\cite{McB}, namely it satisfies the following property:
$$|{\rm Aut}(\Phi_{4})| = 84(g-1) = 168.$$ 
It is well-known that the automorphism group ${\rm Aut}(\Phi_{4})$ is realized by a subgroup $G \subset {\rm Aut}(\mathbb{P}^{2}_{\mathbb{C}}) = {\rm PGL}(3,\mathbb{C})$. 

Somehow surprisingly, there exists a very symmetric arrangement of $21$ lines $\mathcal{K}$ in $\mathbb{P}^{2}_{\mathbb{C}}$ that is determined by the group $G$. These lines are invariant under the action of the group $G$ and it turns our that $\mathcal{K}$ has exactly $21$ quadruple and $28$ triple intersection points. The symmetricity of arrangement $\mathcal{K}$ is manifested by the distribution of intersection points, namely on each line we have exactly $4$ quadruple and $4$ triple intersection points. In terms of the theory of configurations~\cite{Gru, PS}, the $21$ lines and $21$ quadruple points in the Klein arrangement form a symmetric $(21_{4})$ point-line configuration. Let us recall that, according to the classical Sylvester--Gallai theorem, the Klein arrangement $\mathcal{K}$ cannot be realized over the real numbers since it does not possess any double intersection point. On the other hand, in the $1980$-ties, Gr\"unbaum and Rigby constructed a real symmetric $(21_{4})$ point-line configuration that is based on the geometry of the Klein quartic curve~\cite{GR}. The main aim of the present paper is to provide a detailed comparison (and some geometric explanations) between the Klein arrangement of $21$ complex lines and the $21$ lines constructed by Gr\"unbaum and Rigby. Based on that, we explain how to construct, using the geometry of both constructions, arrangements of smooth conics.

\section{Klein's arrangement and the Gr\"unbaum--Rigby configuration}

\subsection{Klein's arrangement of lines} 
\label{kline}
We start with a short outline regarding Klein's approach towards the construction of his arrangement.
Let $G = {\rm PSL}(2,7) \cong {\rm Aut}(\Phi_{4})$ be the unique simple group of order $168$ and consider the action of $G$ on the homogeneous coordinate ring $S = \mathbb{C}[x,y,z]$ of $\mathbb{P}^{2}_{\mathbb{C}}$ - recall that this is due to the fact that $G$ arises naturally as a subgroup of ${\rm PGL}_{3}(\mathbb{C})$ of automorphisms of $\mathbb{P}^{2}_{\mathbb{C}}$. The ring of invariants $S^{G}$ of polynomials invariant under the action of $G$ is well-known since Klein's original work from 1878. The ring of invariants is generated by four polynomials $\Phi_{4}$, $\Phi_{6}$, $\Phi_{14}$, and $\Phi_{21}$, where the subscript denotes the degree. The most important invariant polynomial, from the perspective of our work, is $\Phi_{21}$ which is the defining equation of the Klein arrangement $\mathcal{K}$ of $21$ lines. The polynomials $\Phi_{4}, \Phi_{6}, \Phi_{14}$ are algebraically independent, and $\Phi_{21}^{2}$ belongs to the ring generated by $\Phi_{4}$, $\Phi_{6}$, and $\Phi_{14}$. Now we explain a bit how the generators of $S^{G}$ look like. The polynomial $\Phi_{6}$ is defined by the determinant of the Hessian of $\Phi_{4}$, and it has the following form:
$$\Phi_{6} = xy^{5} + yz^{5} + zx^{5} - 5x^{2}y^{2}z^{2}.$$
The degree $14$ part of $S^{G}$ is spanned by $\Phi_{14}$ and $\Phi^{2}_{4}\Phi_{6}$, so the invariant $\Phi_{14}$ is not uniquely determined up to constants. One possible choice is 
$$\Phi_{14} = \frac{1}{9} {\rm BH}(\Phi_{4}, \Phi_{6}),$$
where ${\rm BH}(\Phi_{4},\Phi_{6})$ is the so-called border Hessian, i.e., 
$${\rm BH}(\Phi_{4},\Phi_{6}) := \begin{vmatrix} \partial^2 \Phi_{4}/\partial x^2&\partial^2 \Phi_{4}/\partial x \partial y&\partial^2 \Phi_{4}/\partial x \partial z & \partial \Phi_{6} / \partial x\\
\partial^2 \Phi_{4}/\partial y \partial x&\partial^2 \Phi_{4}/\partial y^2&\partial^2 \Phi_{4}/\partial y \partial z & \partial \Phi_{6}/\partial y\\
\partial^2 \Phi_{4}/\partial z\partial x&\partial^2 \Phi_{4}/\partial z \partial y&\partial^2 \Phi_{4}/\partial z^2 & \partial \Phi_{6}/\partial z\\
\partial \Phi_{6}/\partial x & \partial \Phi_{6}/\partial y & \partial \Phi_{6}/\partial z & 0 \end{vmatrix}.$$
Finally, the polynomial $\Phi_{21}$, the defining equation of the Klein arrangement, is given by the following Jacobian determinant
$$\Phi_{21} =\frac{1}{14} J(\Phi_{4},\Phi_{6},\Phi_{14}) = \frac{1}{14} \begin{vmatrix} \partial \Phi_{4}/\partial x & \partial \Phi_{4}/\partial y & \partial \Phi_{4}/\partial z\\
\partial \Phi_{6}/\partial x & \partial \Phi_{6}/\partial y & \partial \Phi_{6}/\partial z\\
\partial \Phi_{14}/\partial x & \partial \Phi_{14}/\partial y & \partial \Phi_{14}/\partial z\\ \end{vmatrix}.$$

After presenting the way to construct the defining equation of the Klein
arrangement of lines, we focus on the singular locus of $\mathcal{K}$. It is
classically known that the Klein arrangement $\mathcal{K}$ contains $28$ triple
and $21$ quadruple points. We know that the set of all triple points forms an
orbit $\mathcal{O}_{28}$, the set of all quadruple points forms an orbit
$\mathcal{O}_{21}$, and the invariant curves defined by $\Phi_{4}$ and
$\Phi_{6}$ meet in $24$ points forming an orbit $\mathcal{O}_{24}$. Moreover,
the $21$ quadruple points are dual to the $21$ lines in the Klein arrangement,
and $28$ triple intersection points are dual to $28$ bitangents of the Klein
quartic $\Phi_{4}$.

Let us point out here that the defining equation of $\Phi_{21}$ is rather
complicated, so we present below better looking equations (after an appropriate
change of coordinates) that we are going to work with. The defining equation
for $\mathcal{K}$ can be realized over $\mathbb{R}[\sqrt{-7}]$, see for
instance \cite{Steen}, and we are going to use this fact. 
\newpage

The equations of the lines look as follows - here we denote by $a$ a complex root of $x^2+x+2$. 

\[
\begin{tabu}{llllll}
\ell_{1}: & \, x  = 0, & \ell_{2}: & \, y+z =0, & \ell_{3}: & \, ax + y - z = 0, \\
\ell_{4}: & \, ax - y + z = 0, & \ell_{5}: & \, -y + z = 0, & \ell_{6}: & \, x + ay -
z = 0, \\
\ell_{7}: & \, ax - y - z = 0, & \ell_{8}: & \, -x + y + az = 0, & \ell_{9}: & \, ax +
y + z = 0, \\
\ell_{10}: & \, -x + ay + z = 0, & \ell_{11}: & \, -x - y + az = 0, & \ell_{12}: & \,
x + z = 0, \\
\ell_{13}: & \, -x + ay - z = 0, & \ell_{14}: & \, x + ay + z = 0, & \ell_{15}: & \,
-x + z = 0, \\
\ell_{16}: & \, x - y + az = 0, & \ell_{17}: & \, x + y + az = 0, & \ell_{18}: & \, z
= 0, \\
\ell_{19}: & \, -x + y = 0, & \ell_{20}: & \, x + y = 0, & \ell_{21}: & \, y = 0.
\end{tabu}
\]
\bigskip

Now we present projective coordinates of points of intersection of the 21 lines:
\[
\begin{tabu}{lll}
& \mbox{Quadruple points} & \\
&&\\
1: (1:0:0) & 2: (1:-a-1:-1) & 3: (1:a+1:1) \\
4: (1:-a-1:1) & 5: (1:-1:1+a) & 6: (0:1:1) \\
7: (a+1:-1:1) & 8: (0:0:1) & 9: (a+1:-1:-1) \\
10: (a+1:1:1) & 11: (1:0:-1) & 12: (1:1:a+1) \\
13: (1:-1:0) & 14: (a+1:1:-1) & 15: (1:a+1:-1) \\
16: (1:0:1) & 17: (0:1:-1) & 18: (1:1:0) \\
19: (1:-1:-1-a) & 20: (1:1:-1-a) & 21: (0:1:0) \\
&&\\
& \mbox{Triple points} & \\
&&\\
22: (0:1:a) & 23: (0:1:-a) & 24: (a:-1:0) \\
25: (0:a:-1) & 26: (-a+1:1:-1) &  27: (1: -1: a-1) \\
28: (1: 1-a:1) & 29: (a:1:0) & 30: (1:-a:0) \\
31: (1: -1:1-a) & 32: (1:1:a-1) & 33: (1:0:a) \\
34: (1:a:0) & 35: (1:-1:1) & 36: (1:a-1:1) \\
37: (0:a:1) & 38: (a-1:1:1) & 39: (a:0:1) \\
40: (1:1:-1) & 41: (1:1:1) & 42: (1:0:-a) \\
43: (1:1:-a+1) & 44: (a-1:1:-1) & 45: (1:-a+1:-1) \\
46: (1:a-1:-1) & 47: (a-1:-1:-1) & 48: (1:-1:-1) \\
49: (a:0:-1) & & \\
\end{tabu}
\]
The first 21 points are the points where 4 lines meet. The remaining 28 points
are where exactly 3 lines meet. Moreover, each line contains 4 of the first 21 points
and 4 of the last 28 points.

Now we would like to focus on a certain subarrangement contained in the Klein arrangement. Using the notation above, consider the following set of lines: 

\begin{equation} \label{eg:K'}
\mathcal{K}' = \{\ell_{3},\ell_{4},\ell_{6},\ell_{7},\ell_{8},\ell_{9},\ell_{10},\ell_{11},\ell_{13},\ell_{14},\ell_{16},\ell_{17}\}.
\end{equation}
  
It turns out that $\mathcal{K}'$ has interesting properties which can be summarised as follows.
\begin{proposition} \label{thm:K'}
The lines in $\mathcal{K}'$ intersect at $12$ triple and $30$ double points. Moreover, one each line from $\mathcal{K}'$ there are exactly $3$
triple intersection points. In other words, $\mathcal{K}'$ yields a 
complex $(12_{3})$ point-line configuration.
\end{proposition}
\begin{proof}
Our proof is computer-assisted. Denote by $Q$ the defining equation of
$\mathcal{K}'$. The singularities of the arrangement can be derived from an
analysis of the Jacobian ideal $$J_{Q} = \bigg\langle \frac{\partial Q}{\partial x}, \frac{\partial Q}{\partial y}, \frac{\partial Q}{\partial z} \bigg\rangle .$$ One can check that 
$$
{\rm deg}(J_{Q}) := \sum_{p \in {\rm Sing}(\mathcal{K}')} 
({\rm mult}_{p}(\mathcal{K}') -1)^{2}= 78,
$$ 
where ${\rm mult}_{p}(\mathcal{K}')$ is the multiplicity of a singular point
$p$, i.e., the number of lines intersecting at $p$. If we consider $H_{Q}$,
which is defined as the ideal consisting of all partial derivatives of the
second order, then ${\rm deg}(H_{Q}) = 12$, and if we define by $T_{Q}$ the
ideal generated by all the partial derivatives of order three, then $T_{Q}$
defines over $\mathbb{C}$ the empty set. This description justifies that we
have exactly $12$ triple and $30$ double intersection points. 

Using \verb}Singular}, we can also determine the coordinates of all
intersection points of multiplicity $3$, namely:
\[ \begin{tabu}{ll} 
1: (a-1:a-1:-a-3) & 2: (a+1:-a+1:-a-1) \\
3: (a-1:a+1:-a-1) & 4: (-a+1:a+1:-a-1) \\
5: (-a-1:a+1:a-1) & 6: (-a-1:-a+1:-a-1) \\
7: (-a-1:-a+1:a+1) & 8: (a-1:-a+1:-a-3) \\
9: (a-1:-a-1:-a-1) & 10: (-a+1:-a-1:-a-1) \\
11: (a+1:a+1:-a+1) & 12: (-a-1:a-1-a-1) 
\end{tabu}\]

Now it is easy to check that $\mathcal{K}'$ together with the points listed above form a symmetric $(12_{3})$ point-line configuration, which completes the proof.
\end{proof}

Let us denote the above $(12_3)$ point-line configuration by $K'(12_3)$.
\begin{remark}
As was kindly pointed out by the referee, if we take the remaining $9$ lines from the Klein arrangement, i.e. only those lines with purely real coefficients, then we get a real simplicial arrangement. We do not know whether this observation was known, but it seems to be quite interesting. 
\end{remark}
\subsection{The $(21_4)$ Gr\"unbaum--Rigby configuration}

The $21$ lines and $21$ quadruple intersection points of $\mathcal K$ given in
the previous section form a balanced $(21_4)$ configuration. We note that,
following Gr\"unbaum~\cite{Gru}, we use the more recent term \emph{balanced} for
a configuration consisting of an equal number of points and lines, instead of
the elder term \emph{symmetric} which is now reserved for a configuration
exhibiting some non-trivial geometric symmetry. We denote this configuration by
$\mathrm{K}(21_4)$.

Since the time of Klein's paper, several different descriptions of
$\mathrm{K}(21_4)$ occurred in the literature, even in terms of finite geometry
(see~\cite{GR} for some details). The first geometric realization over the real
numbers is due to Gr\"unbaum and Rigby~\cite{GR}. Their drawing is reproduced in
Figure~\ref{fig:GR_Labels}, apart from the labelling which is ours and uses the 
notations of points and lines of $\mathcal{K}$ given in the previous subsection.
This labelling verifies that the Gr\"unbaum--Rigby configuration is isomorphic
to $\mathrm{K}(21_4)$ -- let us recall that two configurations $C'$ 
and $C''$ are isomorphic if there is a bijective correspondence between the set of points of $C'$ and $C''$ and between the set of lines of $C'$ and $C''$ such that incidences are preserved, see e.g.~\cite[p.\ 18]{Gru}.

Thus we can see that the Gr\"unbaum--Rigby configuration forms 
a geometric point-line realization of $\mathrm{K}(21_4)$ over the real numbers. 
We shall use the notation $\mathrm{GR}(21_4)$ for this realization. Equations of its lines can be given in the following form:
\begin{align*}
A_j&:\quad (\sin 2j\varphi)x - (\cos 2j\varphi)y + z =0,\\
B_j&:\quad (\sin (2j+1)\varphi)x - (\cos (2j+1)\varphi)y - pz =0,\\
C_j&:\quad (\sin 2j\varphi)x - (\cos 2j\varphi)y - qz =0,
\end{align*}
where $j\in \mathbb Z_7$, $\varphi = \pi/7$, $p=(\cos4\varphi)/(\cos2\varphi)$, and
$q=(\cos2\varphi -\cos\varphi)/2(\cos2\varphi)(\cos\varphi)$.

It is worth noting that a realization with points and circles has also been given
in~\cite{GP}.
\begin{figure}[!h]
\begin{center}
\includegraphics[width = .9\linewidth]{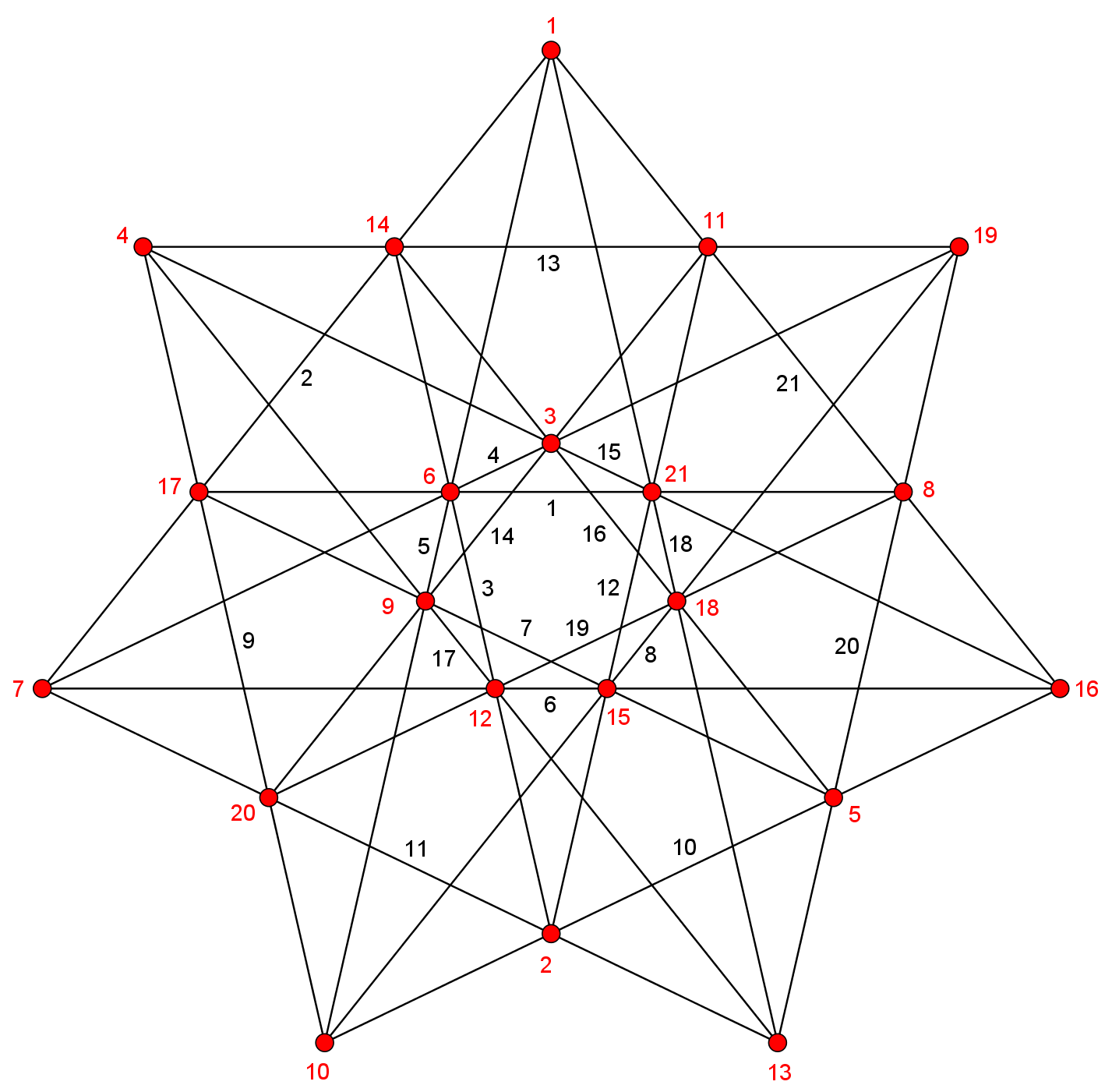}
\caption{The $(21_4)$ Gr\"unbaum--Rigby configuration. The labels of points and
lines are the same as those used above for the quadruple points and for the
lines of $\mathcal{K}$, respectively.}
\label{fig:GR_Labels}
\end{center}
\end{figure}

It is known that the combinatorial symmetry of $\mathrm{GR}(21_4)$ is inherited
from $\Phi_{4}$, i.e.\ its automorphism group is equal to $G$~\cite{Cox83, GR}. 
On the other hand, its geometric symmetry group is isomorphic to the dihedral
group $\mathrm D_7$. Both the set of points and the set of lines decomposes into
three orbits under the action of this group. Moreover, two lines from each of
two orbits of lines pass through each point, and two points from each of two
orbits of points lie on each line. With these symmetry properties, this
configuration belongs to the class of \emph{3-astral}, or, with the more recent
term, \emph{3-celestial} configurations~\cite{Gru,BB}. 

A further interesting geometric symmetry property of $\mathrm{GR}(21_4)$ is that
it is \emph{perfectly self-reciprocal.} We introduce this term for the property
that there is a suitable circle $\Gamma$ such that the polar reciprocity with respect to $\Gamma$ transforms the configuration into itself; this property can be viewed as a strong version of self-duality. We note that reciprocity with respect to a circle 
$\Gamma$ is meant here in the classical sense as it is defined e.g.\ by 
Coxeter and Greitzer~\cite[Section 6.1]{CG}.

To find the circle of reciprocity, in the particular case of $\mathrm{GR}(21_4)$, one proceeds as follows. Due to the 7-fold rotational symmetry, there are three orbits of points, each having a circumcircle; similarly, there are three orbits of lines, each having an incircle. Consider now the outermost point orbit and the innermost line orbit. Construct the \emph{midcircle} of the respective circumcircle and the incircle, i.e.\ the circle $\Gamma$ which determines an inversion swapping the given circumcircle and incircle. Doing the same for the pair of the innermost circumcircle and the outermost incircle, one finds that their midcircle is the same as before; 
moreover, the same circle $\Gamma$ serves also as the midcircle belonging to 
the third pair of point orbit and line orbit. Finally, it is checked that the set 
of points and lines of the configuration decomposes into pairs of pole and polar 
with respect to $\Gamma$.


\subsection{Additional real point-line configurations derived form $\mathcal K$}
\label{sect:additional}

It turns out that the configuration $K'(12_3)$ given in Proposition~\ref{thm:K'}
can be realized as a geometric point-line configuration over the real numbers.
This realization is shown in Figure~\ref{fig:twelve}. The labels of the lines
refer to the set (\ref{eg:K'}), and the labels of the points correspond to the
notation of points of $\mathcal K$. This labelling verifies the isomorphism of
this configuration to $K'(12_3)$, hence it is indeed a realization of
$K'(12_3)$.

Observe that in Figure~\ref{fig:twelve} four additional triple intersection
points occur, namely those of the triples of lines labelled as (7,11,13),
(3,8,14), (4,10,17) and (6,9,16). These extra points of intersection do indeed
exist in this realization (but not in $\mathcal K'$), and adding them to the the
structure one obtains a real point-line $(16_3, 12_4)$ configuration. This
latter configuration is known as the dual of a configuration described by
Zacharias who associated it to an incidence theorem~\cite{Z41, GG19}.
\begin{figure}[!h] 
\begin{center}
    \includegraphics[width=0.7\textwidth]{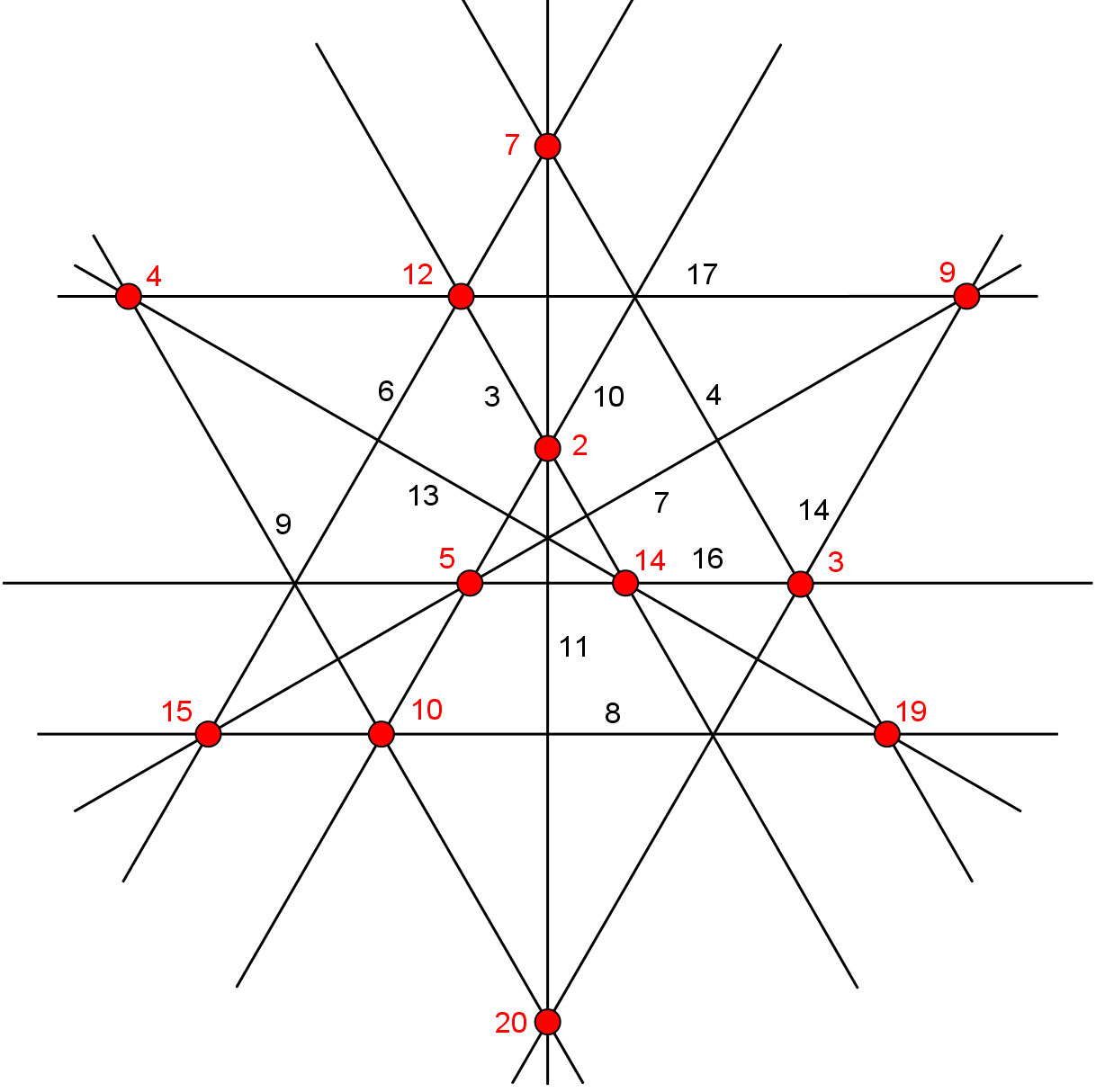} 
    \caption{The real $(12_3)$ point-line configuration derived from $\mathcal{K'}$.}
\label{fig:twelve}
\end{center} 
\end{figure}

\begin{remark}
It is worth emphasizing that the Zacharias configuration of points and lines,
considered as an arrangement of lines, has exactly $19$ triple intersection
points, the maximal possible number of such singular points for $12$ real
lines. Moreover, this configuration has been rediscovered recently in
\cite{BokPok}, where we can find a complete classification of arrangements
with $12$ pseudolines and $19$ triple points.
\end{remark}

In what follows we use certain double points of the lines of $\mathrm{GR}(21_4)$.
As mentioned in the previous subsection, these lines form three orbits under the
action of the $\mathrm D_7$ symmetry group of $\mathrm{GR}(21_4)$. We take double
points which are intersections of two lines belonging to different orbits.
Accordingly, we have three orbits of such points, each consisting of 14 points.
Obviously, they are located on three concentric circles, with radii $r_1 < r_2 <
r_3$; we denote them by $O^{(1)}_{14}$, $O^{(2)}_{14}$ and $O^{(3)}_{14}$,
respectively. Taking them in pairs, we have three different 28-element point sets.
It turns out that in each of these sets there are 21 quadruples of collinear
points. We add to each 28-element point set the lines lying on these quadruples;
in this way we obtain three different point-line configurations of type $(28_3,
21_4$). We denote them by $\mathrm{GR}_{\mathrm d12}(28_3,21_4)$, $\mathrm{GR}_{
\mathrm d23}(28_3,21_4)$ and $\mathrm{GR}_{\mathrm d13}(28_3,21_4)$, where the 
subscripts refer to the pairs of point orbits used in the construction. These 
configurations are depicted in Figures~\ref{fig:cfg12}(a),~\ref{fig:cfg23}(a)
and~\ref{fig:cfg13}(a), respectively.

\begin{remark}
The numerical coincidence in the parameters suggests that some of the
configurations $\mathrm{GR}_{\mathrm d12}(28_3,21_4)$, $\mathrm{GR}_{\mathrm 
d23}(28_3,21_4)$, or $\mathrm{GR}_{\mathrm d13}(28_3,21_4)$ is possibly 
isomorphic to the configuration consisting of the 21 lines and 28 triple 
points of the Klein arrangement. However, we could not establish such an 
isomorphism. Closer examination shows that there are many further point-line
$(28_3,21_4)$ configurations over $\mathbb R$ different from those above, so 
we hope that continuing this work we shall be able to find such a real
representation. 
\end{remark}

Now we take a circle concentric with $\mathrm{GR}_{\mathrm d12}(28_3,21_4)$,
and produce the polar reciprocal 
$\mathrm{GR}_{\mathrm d12}(28_3,21_4)^{\vee}$ with respect to this circle.
This is obviously a $(21_4, 28_3)$ configuration. By scaling the circle of
reciprocity appropriately, we obtain the surprising result that the union
$\mathrm{GR}_{\mathrm d12}(28_3,21_4)\cup \mathrm{GR}_{\mathrm d12}(28_3,21_4)^{
\vee}$ forms a $(49_4)$ configuration (see Figure~\ref{fig:cfg12}(
b). We note that this latter configuration is not simply a union of
$\mathrm{GR}_{\mathrm d12}(28_3,21_4)$ and its reciprocal, since new incidences 
occur here; actually, it is an \emph{incidence sum}, as this notion is defined
in~\cite{BGP}. Thus, when we speak about an appropriate scaling of the circle 
of reciprocity, this means that we change its radius until each point of the 
one configuration will be incident with a corresponding line of the other
configuration, and vice versa.

A direct consequence of this construction is that our new configuration is
perfectly self-reciprocal.

The same procedure can be applied to the two other configurations
$\mathrm{GR}_{\mathrm d23}(28_3,21_4)$ and 
$\mathrm{GR}_{\mathrm d13}(28_3,21_4)$
as well. In this way, we obtain again perfectly self-reciprocal
configurations. We denote these three self-reciprocal examples by 
$\mathrm{GR}_{\mathrm d12}(49_4)$, $\mathrm{GR}_{\mathrm d23}(49_4)$ and
$\mathrm{GR}_{\mathrm d13}(49_4)$, respectively. They are depicted in the (b) part of Figures~\ref{fig:cfg12},~\ref{fig:cfg23} and~\ref{fig:cfg13}.
\begin{figure}[!h]
\begin{center}
\subfigure[]{\hskip -3pt
\includegraphics[width =0.5\linewidth]{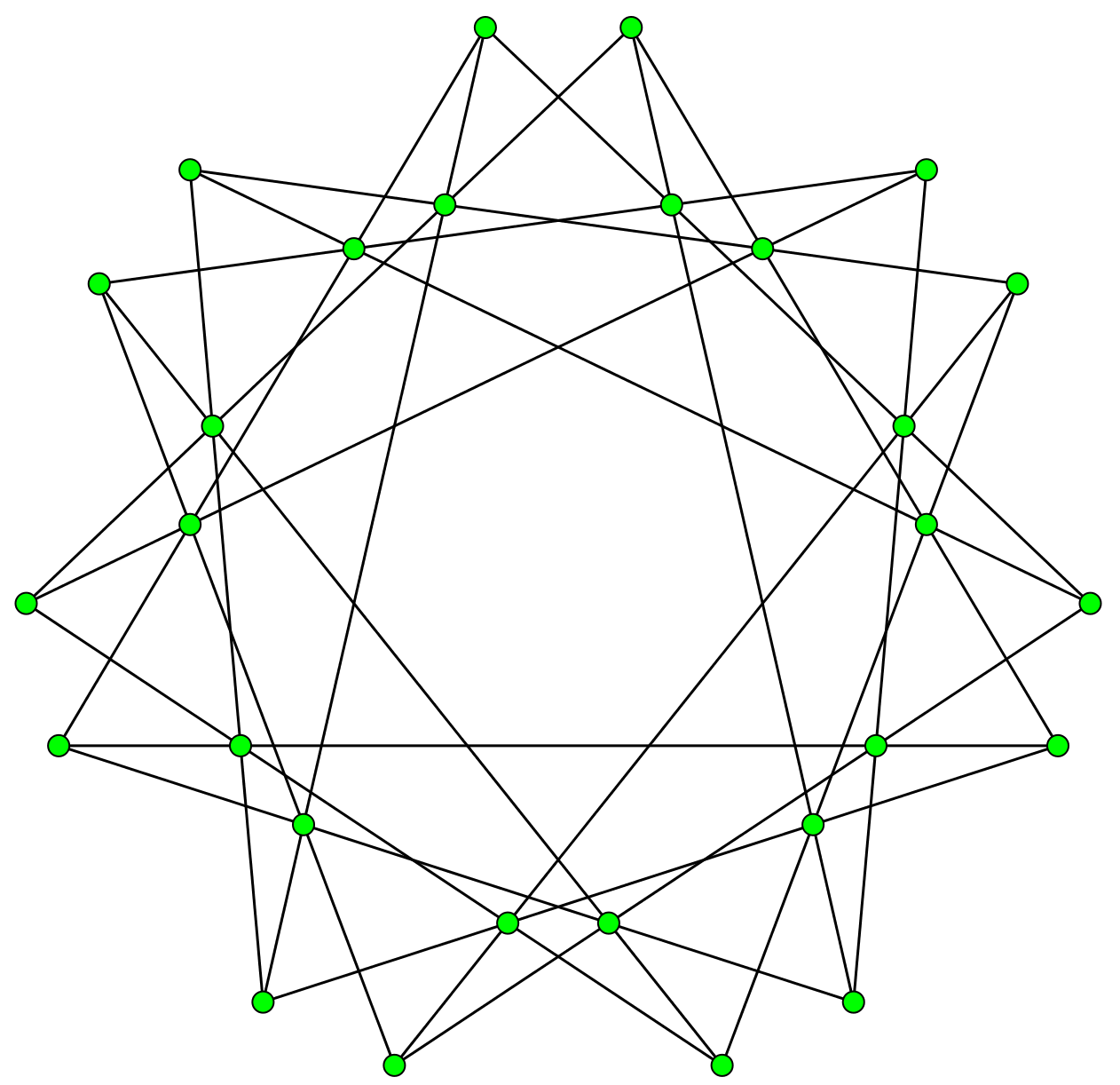}}
\subfigure[]{\hskip -3pt
\includegraphics[width =0.5\linewidth]{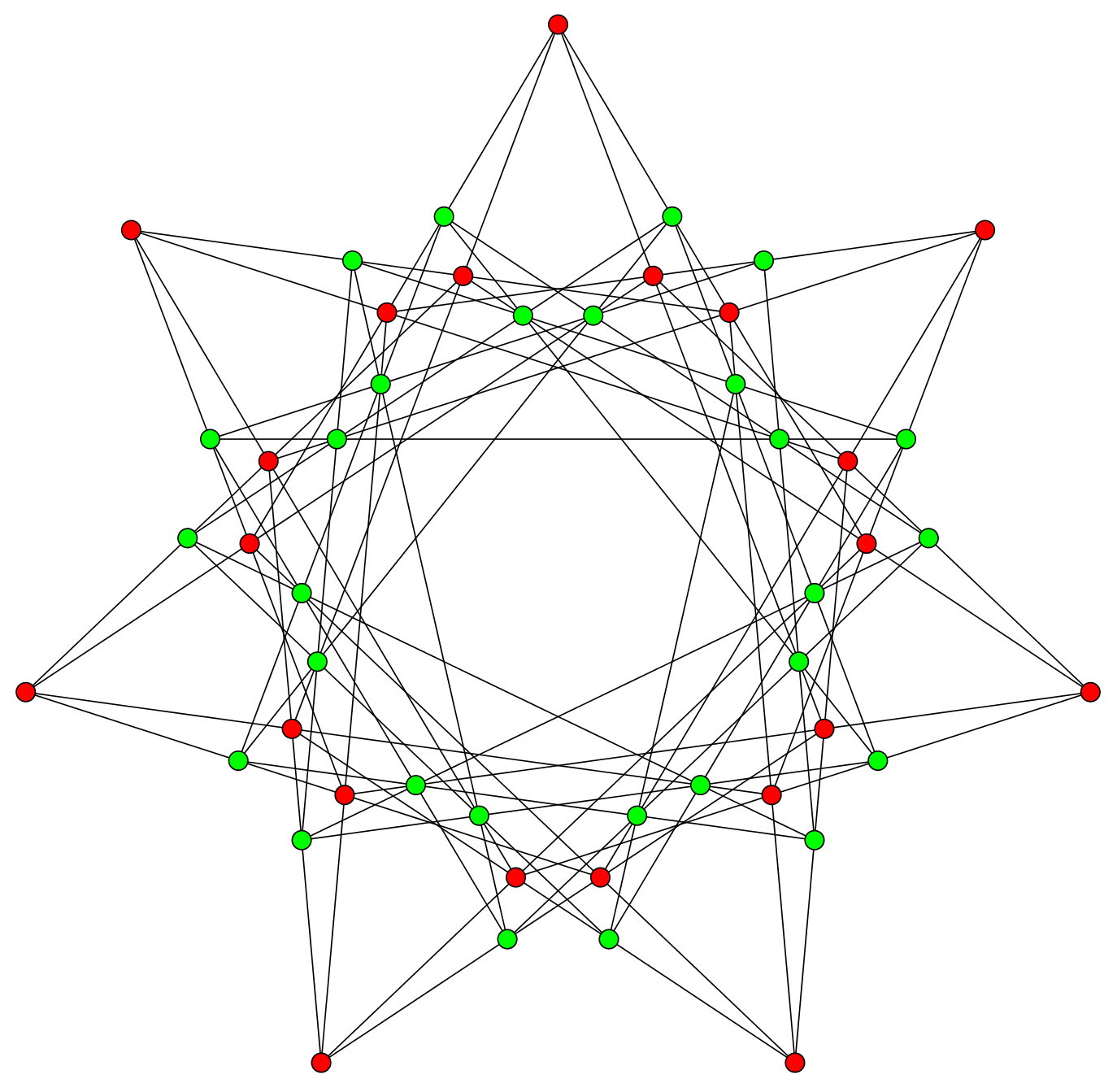}}
\caption{The configuration $\mathrm{GR}_{\mathrm d12}(28_3,21_4)$ (a) and
the incidence sum with its polar reciprocal producing 
$\mathrm{GR}_{\mathrm d12}(49_4)$ (b).}
\label{fig:cfg12}
\end{center}
\end{figure}

\begin{figure}[!h]
\begin{center}
\subfigure[]{\hskip -2pt
\includegraphics[width =0.5\linewidth]{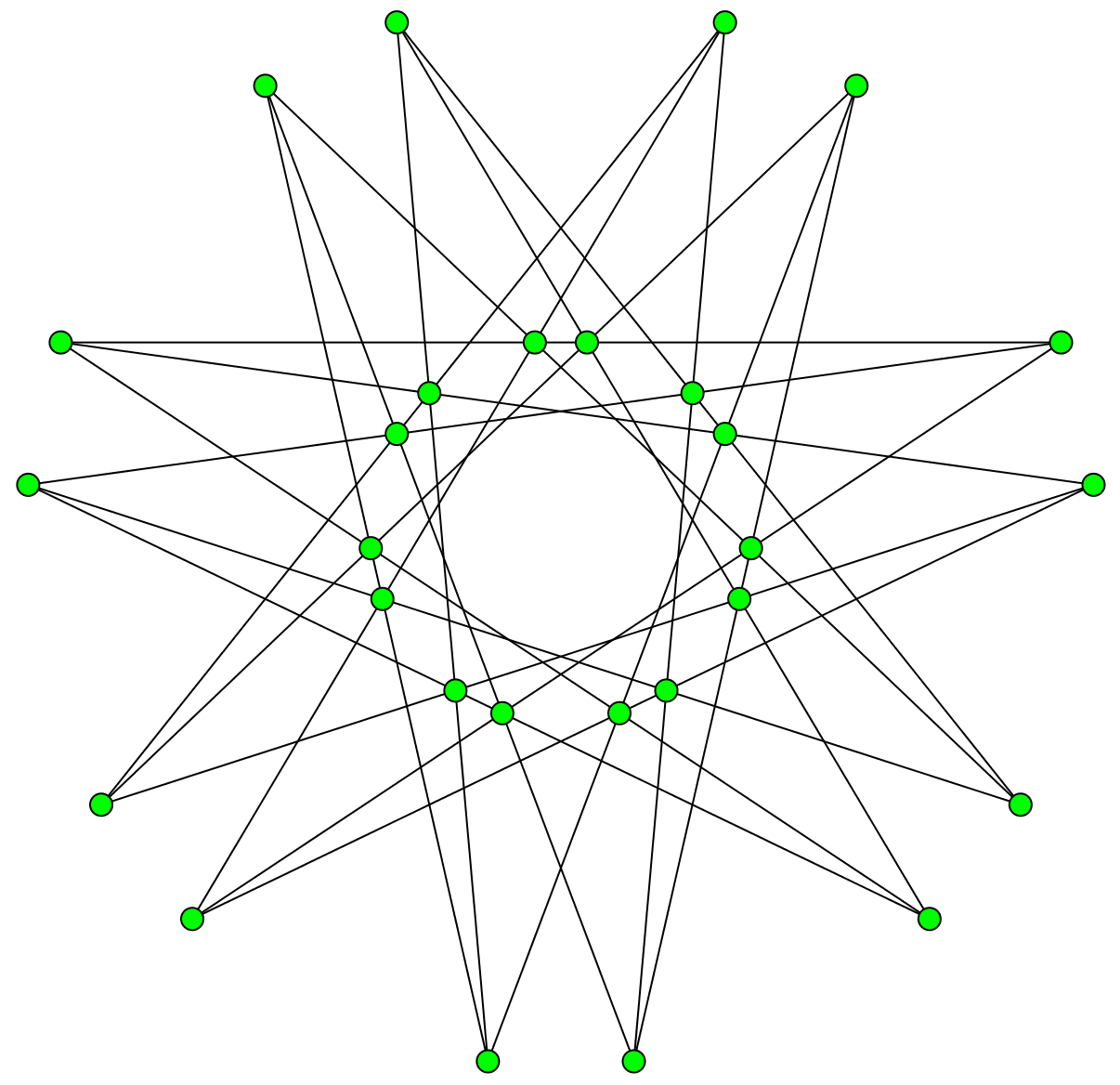}}
\subfigure[]{\hskip -3pt
\includegraphics[width =0.5\linewidth]{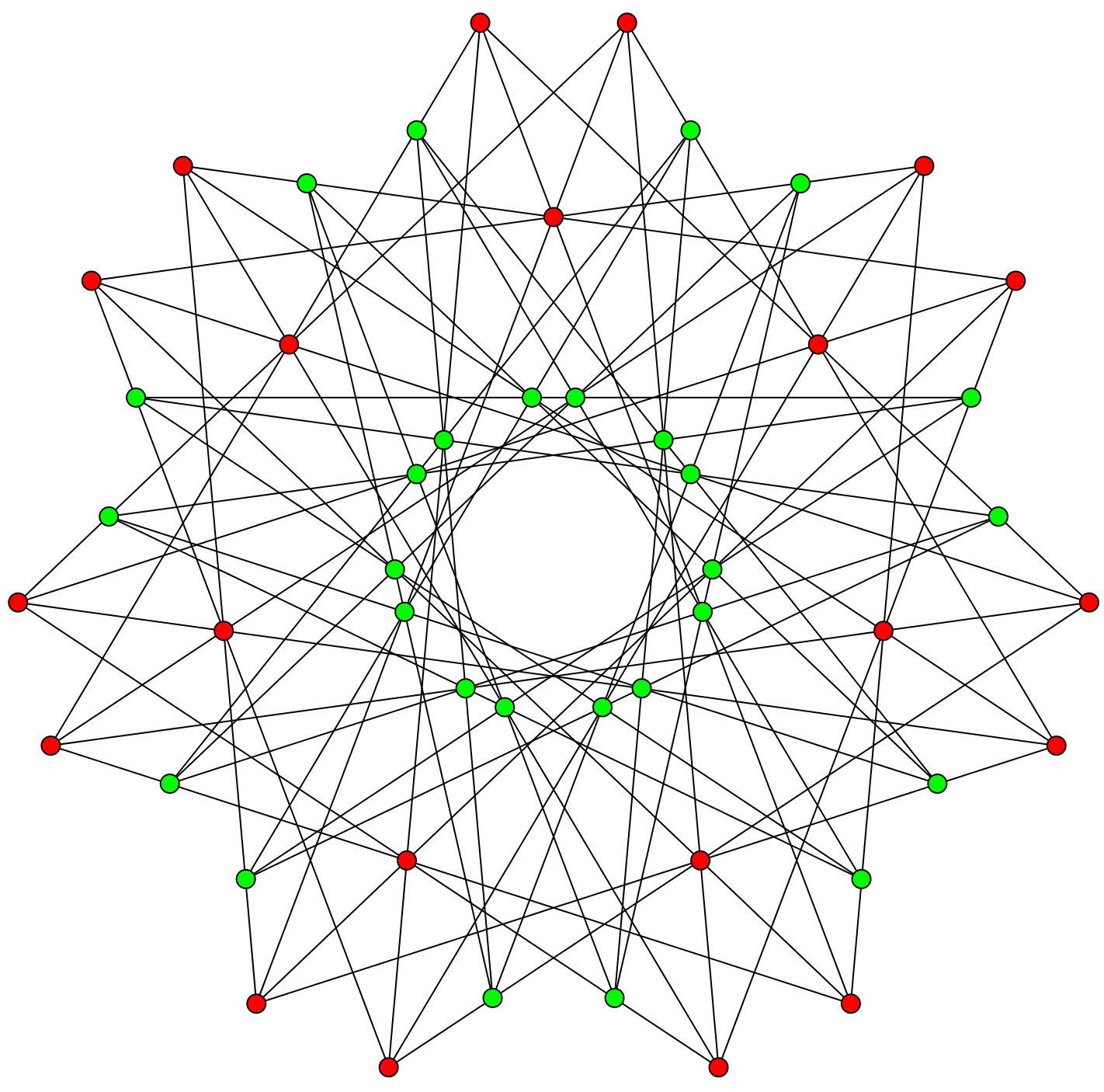}}
\caption{The configuration $\mathrm{GR}_{\mathrm d23}(28_3,21_4)$ (a) and the incidence sum with its polar reciprocal producing $\mathrm{GR}_{\mathrm d23}(49_4)$ (b).}
\label{fig:cfg23}
\end{center}
\end{figure}

\begin{figure}[!h]
\begin{center}
\subfigure[]{\hskip -2pt
\includegraphics[width =0.5\linewidth]{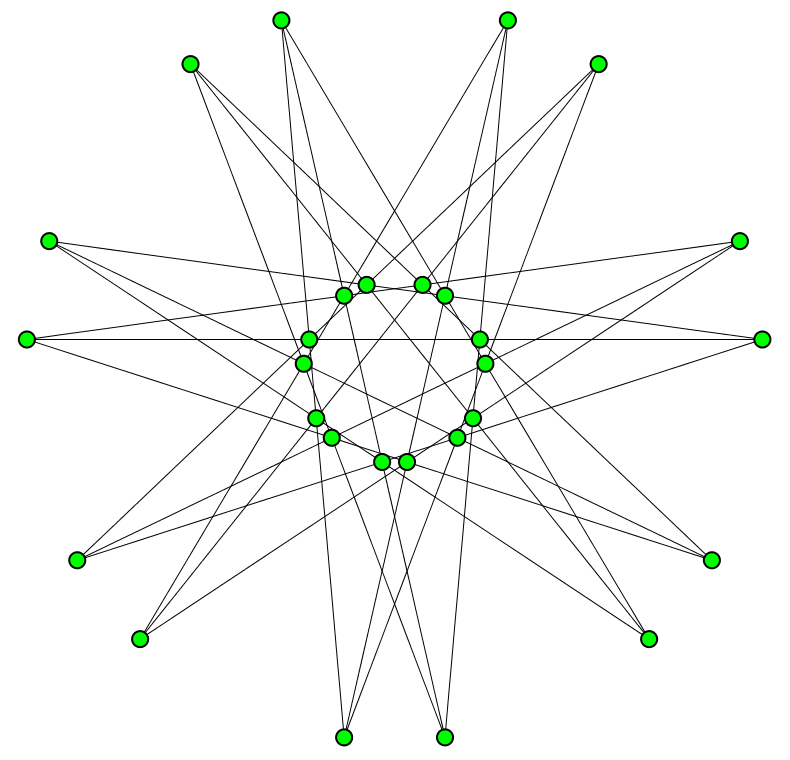}}
\subfigure[]{\hskip -3pt
\includegraphics[width =0.5\linewidth]{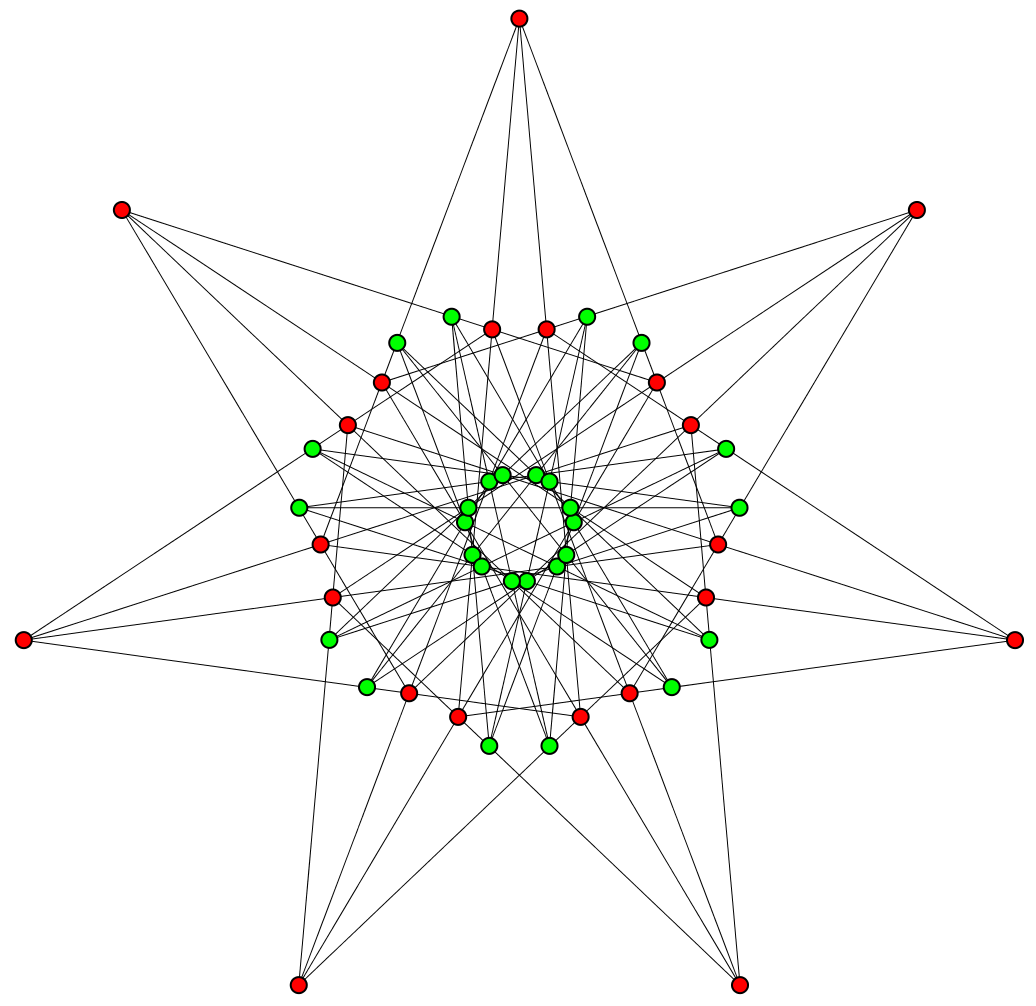}}
\caption{The configuration $\mathrm{GR}_{\mathrm d13}(28_3,21_4)$ (a) and
the incidence sum with its polar reciprocal producing 
$\mathrm{GR}_{\mathrm d13}(49_4)$ (b).}
\label{fig:cfg13}
\end{center}
\end{figure}

Finally, from the orbits $O^{(1)}_{14}$, $O^{(2)}_{14}$ and $O^{(3)}_{14}$,
we can take one half of them such that each forms in itself a 7-orbit, this
time under the action of the rotation group $\mathrm C_7$. Combining these
21 points with the quadruple points of $\mathrm{GR}(21_4)$, one observes
that collinear sixtuples of points occur (with appropriate choice of the
halved orbits). Thus using the corresponding lines, one obtains a $(42_4,
28_6)$ configuration. It is not clear whether it has a closer relationship
with the nodes forming the orbit $\mathcal{O}_{42}$ mentioned in
Proposition~\ref{prop:21ReduciblePolars} below.
\begin{figure}[!h]
\begin{center}
\includegraphics[width = .8\linewidth]{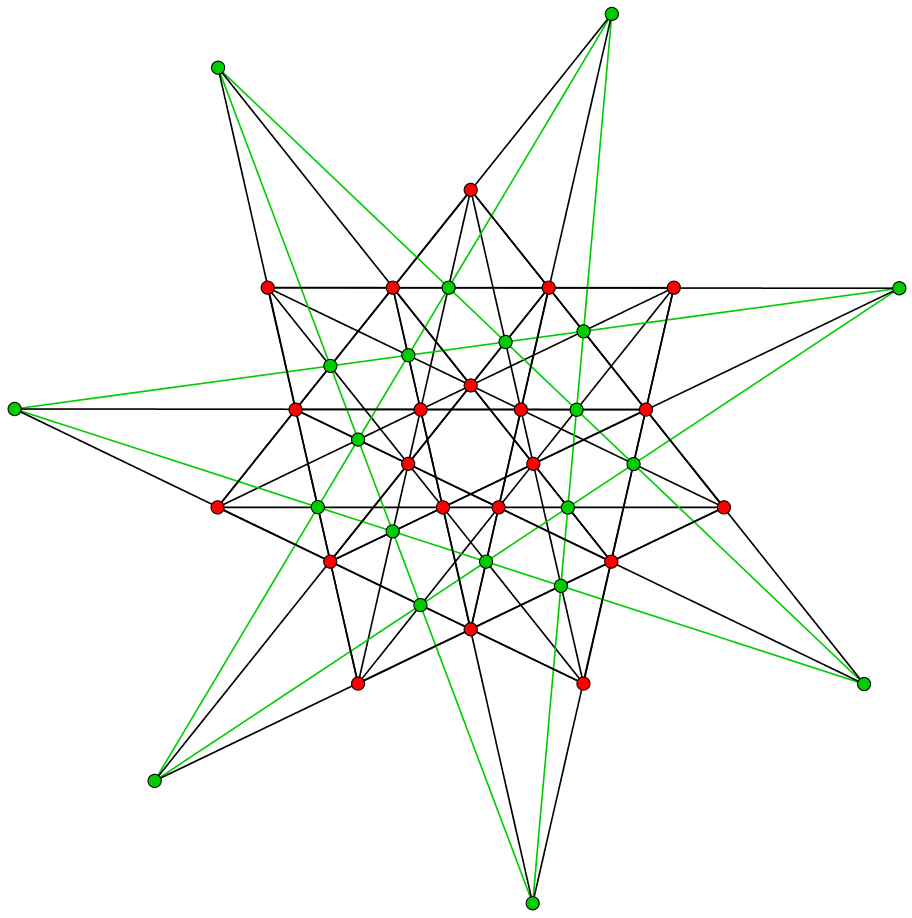}
\caption{A real point-line $(42_4, 28_6)$ configuration derived from $\mathcal{K}$.}
\label{fig:42points}
\end{center}
\end{figure}

\newpage
\subsection{On the rigidity of Klein's arrangement of lines}
\label{sect:rigidity}
Based on our experience regarding the Klein arrangement and Gr\"unbaum-Rigby
configuration, we would like to present here a short comparison between
these two constructions. We point out here some discrepancies that might
explain better why these object should be considered (geometrically)
separately.

We start with the Klein arrangement. The main combinatorial object that can be associated with a line arrangement $\mathcal{A}$ (or hyperplanes, in general) is the intersection lattice $L(\mathcal{A})$ which consists of all intersections of elements of $\mathcal{A}$, including the ambient space as the empty intersection. Now we are going to define the notion of the moduli space associated with arrangements -- we will follow Cuntz's approach presented in \cite{Cuntz}. Let $\mathbb{K}$ be a field and denote by $\mathcal{A} = \{H_{1}, ...,H_{d}\}$ a central hyperplane arrangement in $\mathbb{K}^{3}$. Here by a central arrangement we mean that all hyperplanes pass through the origin and, due to this reason, the projectivization of $\mathcal{A}$ is an arrangement of lines in $\mathbb{P}^{2}_{\mathbb{K}}$. For a matrix $M = [m_{1}, ..., m_{d}] \in M_{3\times d}(\mathbb{K})$ we attach a central hyperplane arrangement $\mathcal{B}_{M} = \{ H_{i}^{'} = {\rm ker}(m_{i}) \, : \, i \in\{1, ...,d\}\}\setminus \mathbb{K}^{3}$ that sits in $\mathbb{K}^{3}$. Now we consider the following condition for $M$:
\begin{eqnarray*}
(\ast) && |\mathcal{B}_M|=d \text{ and there exists an isomorphism } \pi\colon L(\mathcal{A})\rightarrow L(\mathcal{B}_M)\\
&& \text{ of graded lattices such that }
\pi(H_i)=H_i' \text{ for } 1\leq i\leq d .
\end{eqnarray*} 
For a lattice $L$ on $\{1,\ldots,d\}$, we define
$$\mathcal{U}(L)=\{M\in M_{3\times d}(\mathbb{K}) \, : \, M \text{\ satisfies\ }(\ast)\}.$$
Since the condition ($\ast$) is determined by  vanishing / non-vanishing of minors in $M$, it follows that $\mathcal{U}(L)$ has a structure of an algebraic variety.
We define the \emph{moduli space} $\mathcal{V}_{\mathbb{K}}(L) $ of arrangements whose intersection lattice is $L$ as 
\[\mathcal{V}_{\mathbb{K}}(L) := {\rm PGL}(3,\mathbb{K}) \backslash (\mathcal{U}(L) / (\mathbb{K}^*)^d).\] 
For a lattice $L$, we write ${\rm Aut}(L)$ for the set of automorphisms of posets, i.e., the set of bijections preserving the relations in the poset. 

In order to understand the moduli space of the Klein arrangement of lines, we will need one additional definition.
\begin{definition}
Let $L$ be the intersection lattice of a central hyperplane arrangement
$\mathcal{A} \subset \mathbb{K}^{3}$. We call one dimensional elements as
points, and the two dimensional elements as lines (this abusing comes from the
projectivization of $\mathcal{A}$). We say that lines $H_{1}, ..., H_{n}$ in $L$
generate $L$ if there is a sequence of points and lines $U_{1}, ..., U_{m}$ in
$L$ such that:
\begin{enumerate}
    \item[1)] $U_{i} = H_{i}$ for all $i \in \{1, ..., n\}$,
    \item[2)] For all $i > n$ there exist $j,k < i$ such that $U_{i} \in
    \{U_{j}\cap U_{k}, U_{j} + U_{k}\}$,
    \item[3)] Every line of $L$ is contained in $\{U_{1}, ..., U_{m}\}$.
\end{enumerate}
Furthermore, we define
$$g(L) := {\rm min} \{ n \in \mathbb{N} \, : \, L \text{ is generated by } n
\text{ lines}\}.$$
\end{definition}
In other words, our intersection lattice is generated by lines $H_{1}, ...,
H_{n}$ if all the lines in $L$ are obtained by inductively adding intersection
points of two lines, or lines through two points.

Based on the above introduction, we can show the following.

\begin{theorem}
The Klein arrangement $\mathcal{K}$ is a rigid arrangement, i.e., it is unique
up to the projective equivalence.
\end{theorem}
\begin{remark}
The fact that the Klein arrangement is rigid in somehow well-known in folklore. However, we were not able to detect any direct reference towards this property in the literature, and that is the reason why we decided to discuss this issue here.
\end{remark}
\begin{proof} Since our proof of this claim is computer-assisted, based on algorithm \cite[Algorithm 3.9]{Cuntz} due to Cuntz, we present an outline. 
As an input, we take a matroid $L$ of rank three and a field $\mathbb{C}$. The output of the algorithm is a pair of algebraic varieties $V,E$ such that $\mathcal{V}_{\mathbb{C}}(L) = V \setminus E$. We have three steps, in the first one we need to choose a set of generating lines $G:=\{H_{1}, ..., H_{n}\}$ of $L$. Then for a largest subset $S \subset G$ of $G$ in general position, one chooses basis elements of $\mathbb{C}^{3}$ as coordinate vectors and for the remaining $d:=|G\setminus S|$ lines in $G\setminus S$ one chooses coordinate vectors consisting of variables in a polynomial ring $\mathbb{C} [ x_{1}, ..., x_{3d}]$. Since every triple of lines in $L$ gives a condition on the determinant of the corresponding vectors as coordinates, yielding varieties $V$ and $E$ as required depending on whether the determinant is zero or not. According to what we have seen so far, we need to find a set of generating lines. It turns out, based on computer calculations, that for the intersection poset $L$ of the Klein arrangement $\mathcal{K}$ we have 
$$g(L) = 5$$ and the generating lines can be chosen as follows (here we use equations of lines presented in Section \ref{kline}):
\begin{align*}
H_{1} &: \, x  = 0, \\
H_{2} &: \, y  =0, \\
H_{3} &: \, z  = 0, \\
H_{4} &: \, x - (a+1)y - z = 0, \\
H_{5} &: \, x - y + (a+1)z = 0. \\
\end{align*}
Following the path of the algorithm, we can check that indeed the presentation of the Klein arrangement is unique up to projectivities and conjugates.
\end{proof}
\begin{remark}
It was announced in \cite[Example 3.2]{Cuntz} that for reflection arrangements $\mathcal{A}$ associated with irreducible complex reflection groups of rank three the moduli spaces $\mathcal{V}_{\mathbb{C}}(L(\mathcal{A}))$ are finite and these finitely many realizations of $L(\mathcal{A})$ in $\mathcal{V}_{\mathbb{C}}(L(\mathcal{A}))$ are Galois conjugate under automorphisms of the smallest field extension  of $\mathbb{Q}$ over which $L(\mathcal{A})$ is realizable. This explains, in particular, that the Wiman arrangement of $45$ lines with $120$ triple, $45$ quadruple, and $36$ quintuple points is unique up to the projective equivalence.
\end{remark}


\subsection{An incidence conjecture for $\mathrm{GR}(21_4)$}

Observe that $\mathrm{GR}(21_4)$ is the union of three subconfigurations of type
$(14_2, 7_4)$ (see Figure \ref{fig:subcfg}). The set of points of the first
subconfiguration, which we denote by $\mathrm S_{12}$, is composed of the
outermost and the intermediate point orbit of $\mathrm{GR}(21_4)$; we denote
these orbits by $\mathrm P_1$ and $\mathrm P_2$, respectively. Its lines form
the side lines of a regular star-heptagon of type $\big\{\frac{7}{2}\big\}$ (for
this notation, see Coxeter~\cite{Cox61}). The second and the third
subconfiguration, denoted by $\mathrm S_{13}$ and $\mathrm S_{23}$, have 
$\mathrm P_1 \cup \mathrm P_3$ and $\mathrm P_2 \cup \mathrm P_3$ as set of
points, respectively, where $\mathrm P_3$ denotes the innermost point orbit of
$\mathrm{GR}(21_4)$. The lines of both these subconfigurations form the side
lines of a regular star-heptagon of type $\big\{\frac{7}{3}\big\}$. Thus they
are isomorphic to each other, but are non-isomorphic to $\mathrm S_{12}$. 
\begin{figure}[!h]
\begin{center}\hskip 20pt
\subfigure
{
\includegraphics[width=0.3\textwidth]{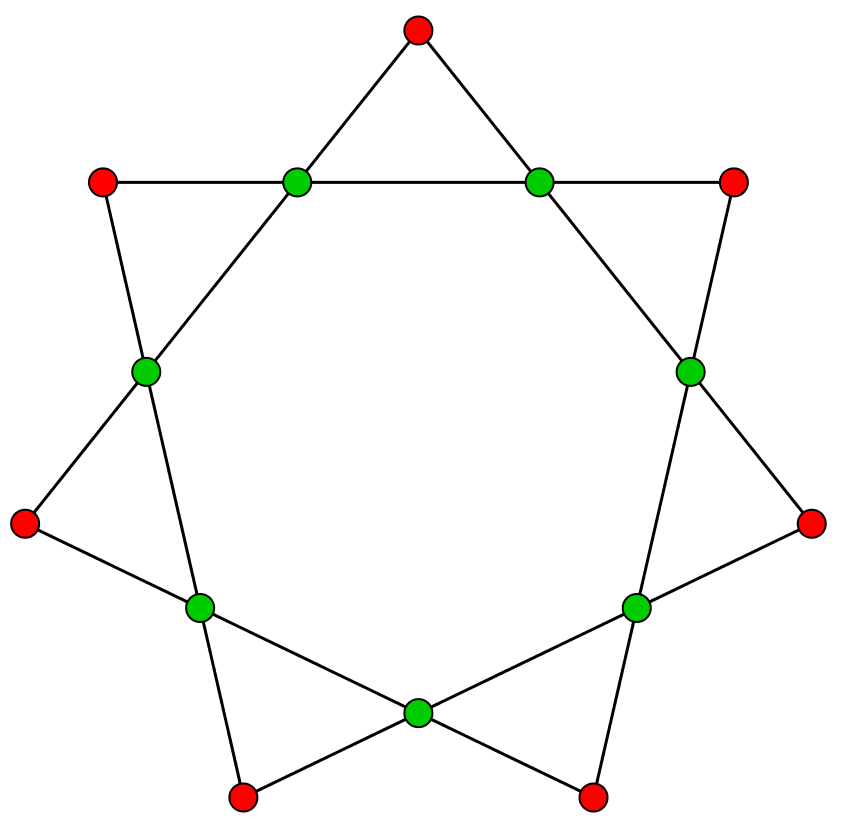}}
\subfigure
{
\includegraphics[width=0.3\textwidth]{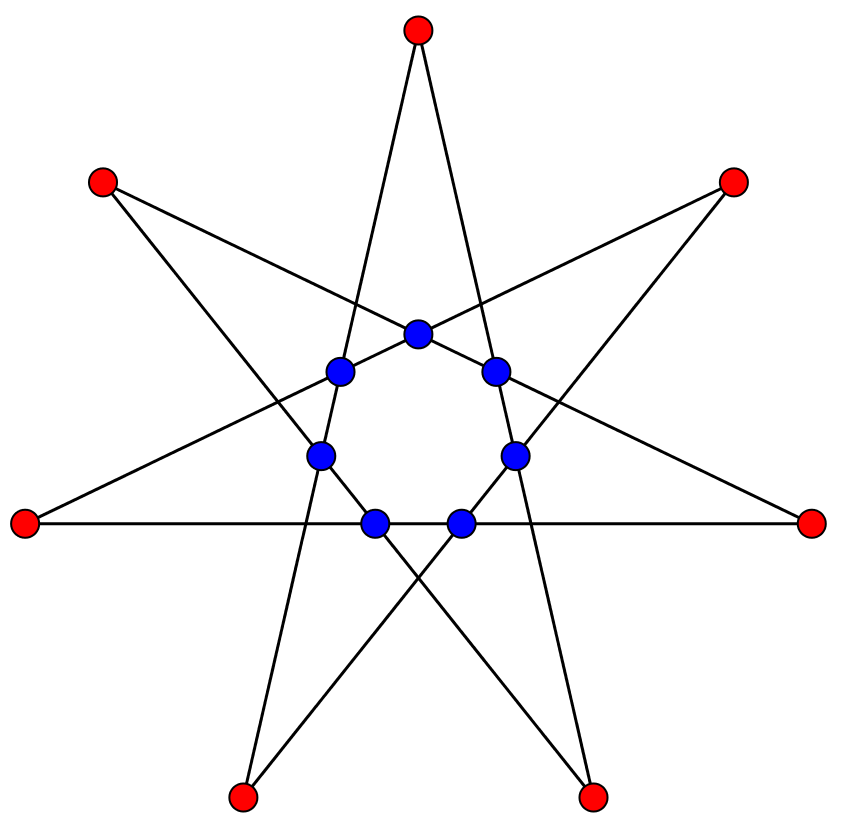}}
\subfigure
{\hskip -6pt
\includegraphics[width=0.3\textwidth]{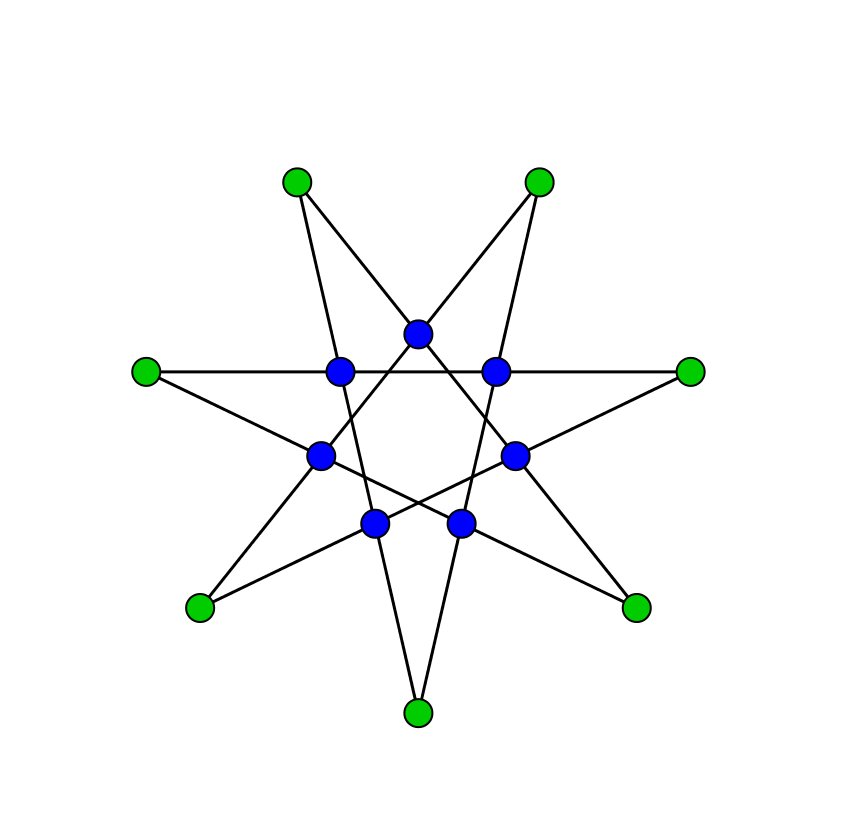}}
\caption{Subconfigurations of type $(14_2, 7_4)$ in the Gr\"unbaum--Rigby configuration. Left: $\mathrm S_{12}$. Centre: $\mathrm S_{13}$. Right: $\mathrm
S_{23}$.
The point orbits $\mathrm P_1$, $\mathrm P_2$, $\mathrm P_3$ are denoted by different colours.}
\label{fig:subcfg}
\end{center}
\end{figure}

Now consider three configurations $\mathrm S'_{12}$, $\mathrm S''_{13}$, and $\mathrm S'''_{23}$, such that the following isomorphisms hold: 
$$ 
\varphi_1: \mathrm S'_{12} \rightarrow \mathrm S_{12}; \quad
\varphi_2: \mathrm S''_{13} \rightarrow \mathrm S_{13}; \quad
\varphi_3: \mathrm S'''_{23} \rightarrow \mathrm S_{23}. 
$$
Under this correspondence, the point sets of the three new configurations are:
$$
\mathrm P'_1, \mathrm P'_2 \subset \mathrm S'_{12}; \quad
\mathrm P''_1, \mathrm P''_3 \subset \mathrm S''_{13}; \quad
\mathrm P'''_2, \mathrm P'''_3 \subset \mathrm S'''_{23}.
$$ 
We have the following conjecture.
\begin{conjecture} \label{conjecture}
Suppose that the point set $\mathrm P'_1$ is inscribed in a conic $C_1$, and the
point set $\mathrm P'_2$ is inscribed in a conic $C_2$. Then $\mathrm S'_{12}$,
$\mathrm S''_{13}$, and $\mathrm S'''_{23}$ can be chosen in such a way that their
union forms a configuration isomorphic to $\mathrm{GR}(21_4)$, where 
\begin{itemize}
\item
$\mathrm P'_1$ is identified with $\mathrm P''_1$,
\item
$\mathrm P'_2$ is identified with $\mathrm P'''_2$,
\item
$\mathrm P''_3$ is identified with $\mathrm P'''_3$;
\end{itemize}
moreover, $\mathrm P''_3=\mathrm P'''_3$ is inscribed in a conic $C_3$.
\end{conjecture}
\begin{remark}
$\mathrm S'_{12}$ does not play a distinguished role in the statement above. In
fact, it can be replaced either by $\mathrm S''_{13}$ or by $\mathrm S'''_{23}$
(followed by appropriate cyclic permutation of these subconfigurations).
\end{remark}

Loosely speaking, the conjecture tells us that the shape of the Gr\"unbaum--Rigby
configuration can change in such a way that if both its outermost and 
intermediate point heptad have a circumconic, then its third point heptad also has
a circumconic. Figure~\ref{fig:illustration} shows an example of such a general
shape, together with the circumconics.

\begin{figure}[!h]
\begin{center}
\includegraphics[width = .75\linewidth]{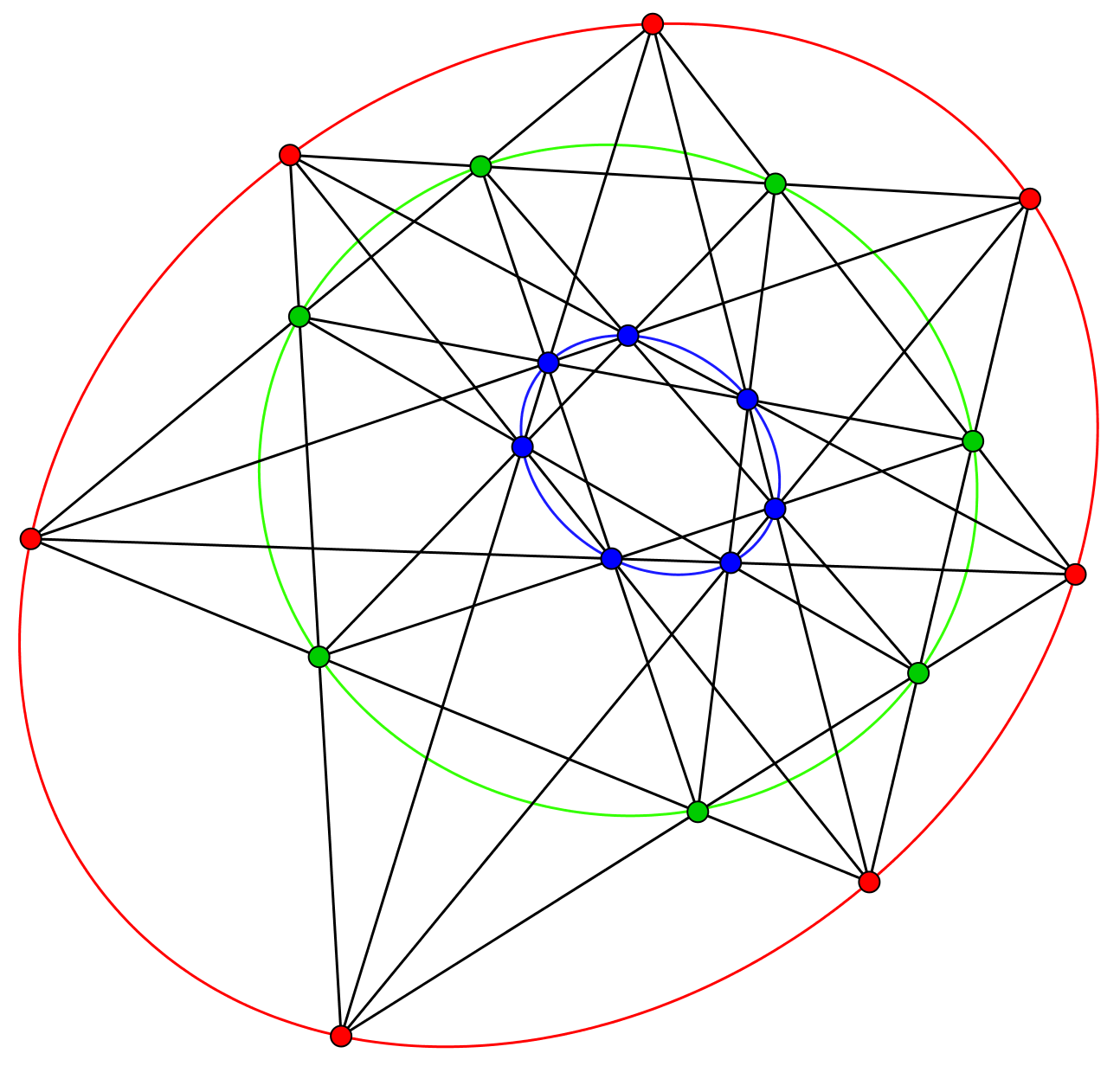}
\caption{Illustration for Conjecture~\ref{conjecture}.}
\label{fig:illustration}
\end{center}
\end{figure}

So far, we have no proof of this conjecture, but by use of dynamic geometry
software, we have experimental evidence. In particular, such experiments show
that there are infinitely many projectively inequivalent versions of this
configuration. This means that if the conjecture is true, then it implies
\emph{movability} of the Gr\"unbaum--Rigby configuration (by the definition
given by Gr\"unbaum~\cite[Section 5.7]{Gru}). In the forthcoming section we deliver an evidence for this claim via the geometry of the moduli space of a certain Gr\"unbaum--Rigby arrangement.
\newpage

\subsection{On the moduli space of a Gr\"unbaum--Rigby arrangement}  
Our prediction, based on computer algebra experiments, tells us that the Gr\"unbaum--Rigby configuration has to be movable in the sense of Gr\"unbaum. By the comments of the referee, which we really appreciate, our aim is to understand one component of the moduli space of realizations of the the Gr\"unbaum--Rigby configuration. In order to do so, we have to fix the whole intersection poset, i.e., this is not only about the incidences between quadruple points and lines, but we also fix other intersection points. More precisely, there exists a geometric realization of a Gr\"unbaum--Rigby configuration that has exactly $21$ quadruple points and $84$ double points. This realization, denoted here by $\mathcal{GR}$, form an arrangement of $21$ lines, and our aim is to proceed our computations for this object. It is worth pointing our here that, by selecting this particular geometric realization, we will find a component of the moduli space parametrizing all possible geometric realization of the intersection poset that we present below. Let us present the incidences that allow us to construct a rank $3$ matroid that is representable over the real and it is associated to $\mathcal{GR}$ -- here $[a,b,c,d]$ means that we have a quadruple points being an intersection of lines labeled by $a,b,c,d$. 
\begin{multline*}
\mathcal{I}_{\mathcal{GR}} = \bigg\{ [ 1, 2, 15, 21 ], [ 1, 3, 8, 13 ], [ 1, 6, 9, 14 ],
[ 4, 5, 15, 16 ], [ 4, 6, 8, 10 ], [ 8, 11, 15, 18 ], [ 9, 13, 17, 21 ], \\ [11, 14, 16, 20 ], [ 2, 4, 9, 11 ], [ 9, 12, 15, 19 ], [ 3, 5, 12, 14 ], [ 8, 12, 17, 20 ], [ 10, 13, 16, 19 ], [ 3, 4, 17, 18 ], \\ [ 2, 3, 19, 20 ], [ 10, 14, 18, 21 ], [ 1, 7, 16, 17 ], [ 5, 6, 20, 21 ], [ 5, 7, 11, 13 ], [ 2, 7, 10, 12 ], [ 6, 7, 18, 19 ]\bigg\}.    
\end{multline*}
Based on the data $\mathcal{I}_{\mathcal{GR}}$ we construct a rank $3$ matroid for which we compute the moduli space. To do this, we followed the algorithmic setup described in \cite{BK}. 
Consider a $3\times n$ matrix $A$ of $3n$ variables $x_{i,j}$ in $\mathbb{C}[x_{i,j}\mid 1\le i \le 3, 1\le j\le n]$.
Then let $I$ be the ideal of all $3\times 3$ minors of $A$ corresponding to triples of lines intersecting in a point.
Similarly, let $J$ be the set of all $3\times 3$ minors of $A$ corresponding to triples of lines that do not intersect in a point.
The moduli space is then $V(I)\setminus \bigcup_{p\in J} V(p)$ as an affine variety in $\mathbb{C}^{3n}$.
To obtain the moduli space of arrangements up to projective equivalences, one can set certain values of the matrix $A$ to $0$ or $1$ depending on the combinatorics of the arrangement (cf. \cite{BK} for details). It turns out that the component of the moduli space is a Zariski open set of $\mathbb{P}^{3}_{\mathbb{C}}$, but the precise description of this set is rather involving and we have to omit it here. In that spot we would like to warmly thank L. K\"uhne for his help with involving computations.


\section{Klein's arrangements of conics}

\subsection{Polars to the Klein quartic}

Now we would like to construct, using two different methods, Klein arrangements
consisting of smooth conics. The first construction, the classical one, dates
back to Gerbaldi in 1882~\cite{Ger}, and it is based on the so-called polars to
the Klein quartic curve. Now we are going to explain how to extract from this
construction $21$ conics. 

We start with the notion of the Steinerian curve which is the locus of points
$p$ such that the polar to a given curve $C$ with respect to $p$ is singular. By
Bertini~\cite{Ber96}, the Steinerian of a general quartic curve has $21$ nodes and $24$ cusps. Recall that the Steinerian curve is invariant under the action of $G$ and the set of $21$ nodes forms an orbit $\mathcal{O}_{21}$ of the length $21$. Moreover, since the Steinerian curve is an invariant curve of degree $12$, it must be a linear combination of $\Phi_{4}^{3}$ and $\Phi_{6}^{2}$. One can observe that the only linear combination of these forms which vanishes at $\mathcal{O}_{21}$ is $4\Phi_{4}^{3}+\Phi_{6}^{2}$, so this is the equation of the Steinerian.

It turns out that each polar to $\Phi_{4}$ with respect to each point $p \in \mathcal{O}_{21}$ splits as a smooth conic and a line meeting transversely at two points. Now we explain how to achieve this goal.

Consider the following gradient map $\nabla(\Phi_{4})$ given by the partial derivatives of the Klein quartic equation
$$\mathbb{P}^{2}_{\mathbb{C}} \ni (x:y:z) \xrightarrow{9:1} (u:v:w) = (3x^{2}y + z^{3} : 3y^{2}z + x^{3} : 3z^{2}x + y^{3}) \in \mathbb{P}^{2}_{\mathbb{C}}.$$
By definition, the polar $P_{p}(\Phi_{4})$ with respect to the point $p=(a:b:c) \in \mathbb{P}^{2}_{\mathbb{C}}$ is the preimage under $\nabla(\Phi_{4})$ of the line dual to $p$, $V(au+bv+cw) \subset \mathbb{P}^{2}_{\mathbb{C}}$. Since the Klein quartic is non-singular, the partial derivatives never vanish, so the gradient map is defined everywhere. We can sum up the most important features of the obtained arrangement of conics and lines -- see \cite[Proposition 2.1]{PR}.
\begin{proposition} \label{prop:21ReduciblePolars}
The Klein quartic $\Phi_{4}$ has exactly $21$ reducible polars and each consists of a line and a conic meeting transversely at two points. They are the polars with respect to the $21$ quadruple points of $\Phi_{21}$. The $21$ lines which are components of the reducible polars are also the components of $\Phi_{21}$, whereas each of the $21$ conics intersect the Klein curve at the eight points of contact of four bitangents. The nodes of the $21$ reducible polars are all distinct and form the orbit $\mathcal{O}_{42}$.
\end{proposition}
Extracting the $21$ conics from the above, we have the following observation.
\begin{proposition}
The arrangement $\mathcal{K}_{2}$ consisting of the $21$ smooth conics extracted from the $21$ reducible polars has only transversal intersection points as the singularities, namely $224=8 \cdot 21 + 56$ triple and $168 = 21 \cdot 8$ double points.
\end{proposition}

We can find the equations of $21$ conics using the following \verb}Singular} script - we hope that this might be useful to the reader who wants to use these conics for different purposes. 
\begin{verbatim}
ring R=(0,e),(x,y,z),dp; 
minpoly=e6+e5+e4+e3+e2+e+1; 
poly Phi4=x3y+y3z+z3x; 
ideal jf=jacob(Phi4);  
poly Phi6=-det(jacob(jf))/54; 
ideal jfh=jf+ideal(Phi6); 
matrix bh[4][4]=transpose(jacob(jfh)),jacob(Phi6);  
poly Phi14=det(bh)/9; 
poly Phi21=det(jacob(ideal(Phi4,Phi6,Phi14)))/14; // the equation of lines 
map f=R,jacob(Phi4); 
poly Phi63=f(Phi21); 
poly Phi42=Phi63/Phi21; 
factorize(Phi42); // the equations of conics
\end{verbatim}
\begin{remark} 
Based on the construction presented above, we dare to call the arrangement $\mathcal{K}_{2}$ as a polar Klein arrangement of smooth conics.
\end{remark}
\subsection{The dual curve to the Klein quartic and an associated conic arrangement}
Now we pass to the second construction that was presented by Roulleau \cite{Roulleau}. This construction is based on the dual curve $\check{C}$ to the Klein quartic curve. It is known that the degree of $\check{C}$ is $12$, since the Klein quartic curve is smooth, and we verify by Pl\"ucker formulae that $\check{C}$ has $52$ singular points, i.e., it has exactly $28$ nodes $\mathcal{P}_{28}$ corresponding to the $28$ bitangents and $24$ cusps $\mathcal{P}_{24}$ corresponding to the $24$ flex points of the quartic. 

One can check, using computer software like \verb}Singular}, that there exists a set of $21$ conics passing through exactly $8$ element subsets of the set $\mathcal{P}_{28}$, and through each point of $\mathcal{P}_{28}$ there are $6$ conics passing through them. It means, we obtained a $(28_{6},21_{8})$ point-conic configuration. One can also check that these $21$ conics, considered as an arrangement of curves, has additionally $420$ double intersection points. One can easily check that there are no other singular points via the following combinatorial count
$$840 =4\cdot \binom{21}{2} = \sum_{r\geq 2} \binom{r}{2}t_{r} = t_{2} + \binom{6}{2}t_{6} = 420 + 15\cdot 28.$$
\begin{proposition}[Roulleau]
There exists a set of $21$ smooth conics determined by the set $\mathcal{P}_{28}$  of all nodes of the dual curve $\check{C}$ to the Klein quartic $\Phi_{4}$. These $21$ conics form an arrangement having $28$ sixtuple points and $420$ double intersection points.

\end{proposition}
Because the defining equation of the above conic arrangement is quite complicated, we skip it here and refer to \cite{Roulleau1} for necessary computations performed by X. Roulleau in \verb}MAGMA}.


\subsection{Real point-conic configurations derived from $\mathcal K$} \label{sect:realconics}

One can derive two different $(21_7)$ point-conic configurations from the
Gr\"unbaum--Rigby configuration in such a way that conics are circumscribed on
7-tuples of points of $\mathrm{GR}(21_4)$~\cite{GG19}. One of them has the
special combinatorial property that it is \emph{resolvable}, which means
that the set of conics can be partitioned into 7 classes (called
resolution classes) such that within each class the conics partition the
point set of the configuration~\cite[Figure 5]{GG19} (we note that the
configuration $K'(12_3)$ described in Section~\ref{sect:additional} is
also resolvable). This particular version is interesting in the light of
our Conjecture~\ref{conjecture}. Experiments show that when applying the
circumconics of the conjecture and destroying the original $\mathrm D_7$
symmetry, the conics circumscribed on the point 7-tuples of the
configuration are preserved as circumconics; as a result, added the three
new conics, we have a \emph{movable} $(21_8, 24_7)$ point-conic
configuration (provided the conjecture is true). 

A class of point-conic configurations can be derived from the
point-line $\mathrm{GR}_{\mathrm d12}(28_3, 21_4)$ configuration 
presented in Section~\ref{sect:additional}. Taking suitable 8-element
subsets of points of this configuration, one can circumscribe 28
conics on them. In this way we obtain a $(28_8)$ point-conic
configuration (see Figure~\ref{fig:(28_8)}). 
\begin{figure}[h]
\begin{center}
\includegraphics[width = .725\linewidth]{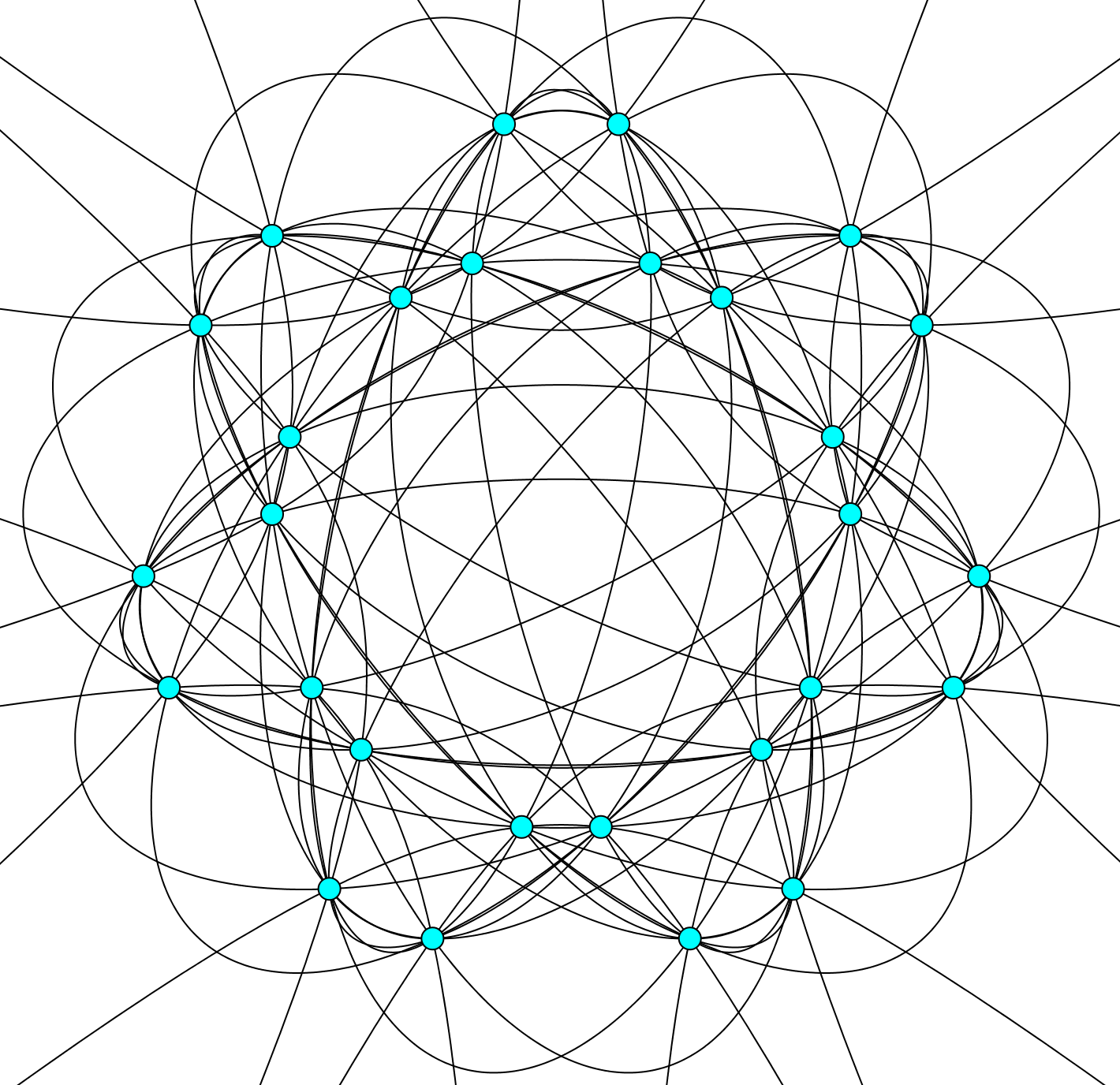}
\caption{A point-conic $(28_8)$ configuration derived from the point-line
configuration $\mathrm{GR}_{\mathrm d12}(28_3, 21_4)$.}
\label{fig:(28_8)}
\end{center}
\end{figure}

The set of conics of this configuration decomposes into four orbits
under the action of the $\mathrm D_7$ geometric symmetry group of
$\mathrm{GR}_{\mathrm d12}(28_3, 21_4)$ (namely, three orbits of
ellipses and one orbit of hyperbolas).  Each orbit of conics forms a
$(28_2, 7_8)$ subconfiguration. A consequence of this combinatorial
regularity of the distribution of points among the conics is that we
have four distinct subconfigurations of type $(28_6, 21_8)$; these are
obtained by removing respectively one of the four orbits of conics. So
far, it is not decided whether they are pairwise non-isomorphic. It is
also subject to further investigation as to whether any of them is
isomorphic to the complex $(28_6, 21_8)$ configuration mentioned in
the previous section.

An additional (balanced) point-conic configuration can be derived in
an analogous way from the point-line configuration 
$\mathrm{GR}_{\mathrm d12}(49_4)$ depicted in Figure~\ref{fig:cfg12}(b). 
Its type is $(49_8)$ (cf. Figure~\ref{fig:(49_8)}). It preserves the 
$\mathrm D_7$ geometric symmetry of the underlying point-line configuration; 
the set of conics is partitioned into 7 orbits (4 of ellipses and 3 of
hyperbolas) under the action of this group. The distribution of the
points among the conics is a little less regular then in the case
above, so subconfigurations of type $(28_2, 7_8)$ are provided only by
three of these orbits (possibly with isomorphic pairs among them). 
\begin{figure}[!h!t]
\begin{center}
\includegraphics[width = .95\linewidth]{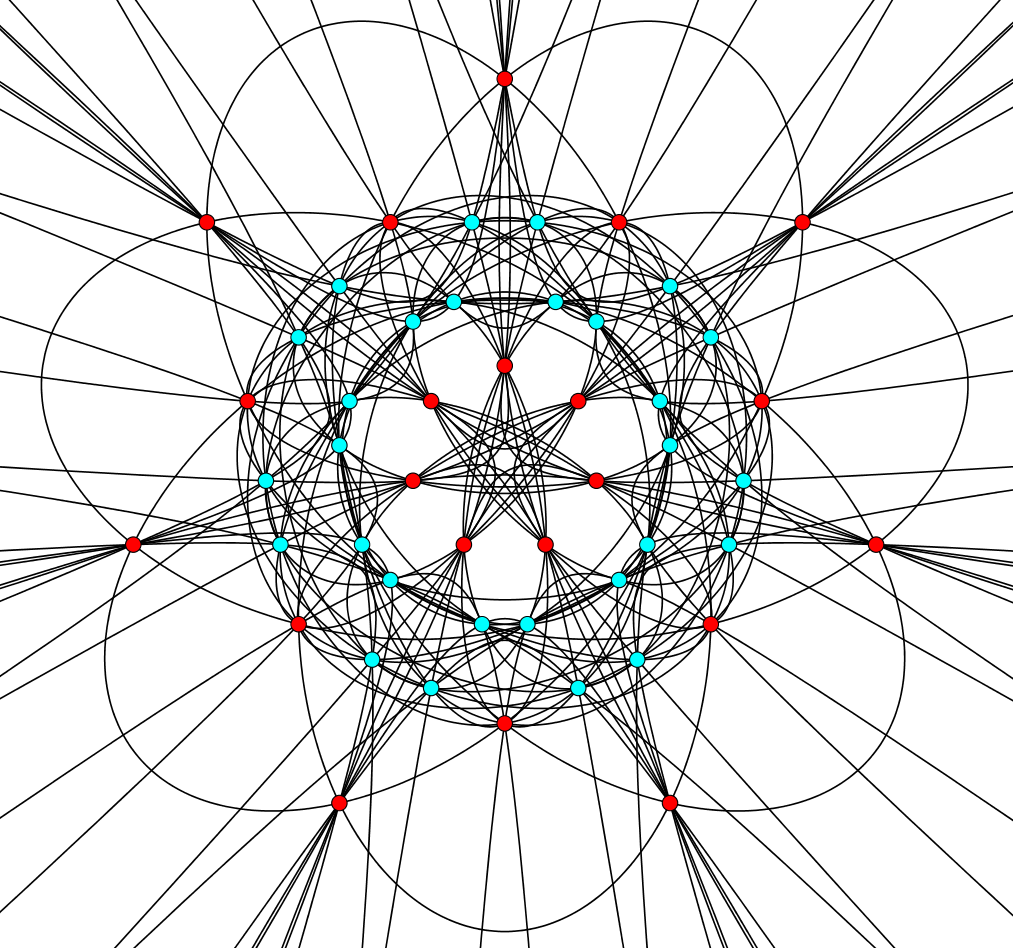}
\caption{A point-conic $(49_8)$ configuration derived from the
point-line configuration $\mathrm{GR}_{\mathrm d12}(49_4)$.}
\label{fig:(49_8)}
\end{center}
\end{figure}

We note that from the non-balanced $(28_2, 7_8)$ subconfigurations considered 
above, one can form 4-factor \emph{Cartesian products}, including repeated use of
any of the factors (for Cartesian products of point-conic configurations
realized in plane, see~\cite{GBKP}). These products form balanced $(614656_8)$
configurations (where $614656=28^4$). 

An additional product can be derived from the $(49_8)$ configuration: omitting
three suitably chosen orbits of conics and 21 points chosen accordingly, one 
obtains in the first step a $(28_4,14_8)$ subconfiguration (see
Figure~\ref{fig:(28_4,14_8)}). Then, in the second step, we take the Cartesian square of this latter (non-balanced) configuration. This provides a balanced $(784_8)$ configuration.
\begin{figure}[!h]
\begin{center}
\includegraphics[width = .75\linewidth]{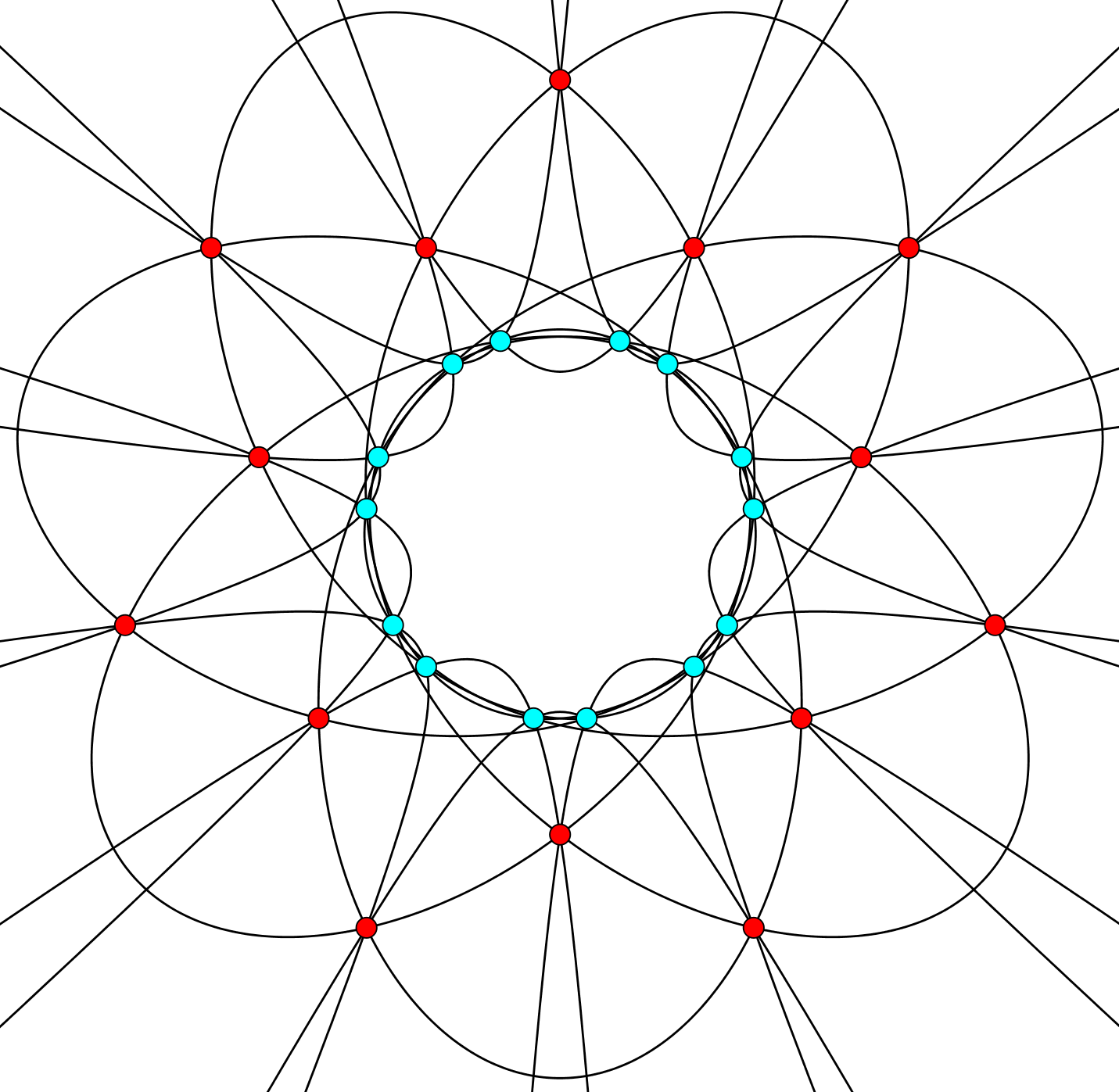}
\caption{The $(28_4, 14_8)$ subconfiguration of the point-conic $(49_8)$ configuration.}
\label{fig:(28_4,14_8)}
\end{center}
\end{figure}



\section*{Acknowledgments}
Both authors would like to thank an anonymous referee for many very useful comments
that allowed to improve the paper, and to Lukas Kuhne for help with symbolic
computations regarding the moduli space of ${\rm GR}(21_{4})$.

G\'abor G\'evay was supported by the Hungarian National Research, 
Development and Innovation Office, OTKA grant No.\ SNN 132625. He also expresses
his thanks to Leah W.\ Berman and Toma\v z Pisanski for the valuable discussions
on Conjecture~\ref{conjecture}.

Piotr Pokora was partially supported by the National Science Center (Poland) Sonata
Grant Nr \textbf{2018/31/D/ST1/00177}.

\vskip 0.5 cm

\newpage

\noindent
G\'abor G\'evay\\
Bolyai Institute,
University of Szeged\\ 
Aradi v\'ertan\'uk tere 1,
H-6720 Szeged, Hungary
\nopagebreak\\
\textit{E-mail address:} \texttt{gevay@math.u-szeged.hu}\\
ORCID: https://orcid.org/0000-0002-5469-5165
\bigskip

\noindent
Piotr Pokora\\
Department of Mathematics,
Pedagogical University of Krakow\\
Podchor\c a\.zych 2,
PL-30-084 Krak\'ow, Poland \\
\nopagebreak
\textit{E-mail address:} \texttt{piotr.pokora@up.krakow.pl}\\
ORCID: http://orcid.org/0000-0001-8526-9831
\bigskip

\begin{thebibliography}{000}
\bibitem{BK} 
M. Barakat and L. K\"uhne, Computing the nonfree locus of the moduli space of
arrangements and Terao?s freeness conjecture, \emph{Math. Comp.}, \textbf{92} (2023),
1431--1452.

\bibitem{BB}
A.\ Berardinelli and L.\ W.\ Berman,
Systematic celestial 4-configurations,
\emph{Ars Math.\ Contemp.}, {\bf 7} (2014), 361--377.

\bibitem{Ber96}
E.\ Bertini, 
Le tangenti multiple della Cayleyana di una quartica piana generale,
\emph{Atti Acad.\ Sci.\ Torino}, {\bf 32} (1896), 32--33.

\bibitem{BGP}
M.\ Boben, G.\ G\'evay and T.\ Pisanski,
Danzer's configuration revisited,
\emph{Adv.\ Geom.} {\bf 15} (2015), 393--408.

\bibitem{BokPok}
J. \ Bokowski and P. \ Pokora,
On the Sylvester-Gallai and the orchard problem for pseudoline arrangements,
\emph{Period. Math. Hung.}, {\bf 77(2)} (2018), 164--174.

\bibitem{Cox61}
H.\ S.\ M.\ Coxeter, 
\emph{Introduction to Geometry}, 
Wiley, New York, 1961.

\bibitem{Cox83}
H.\ S.\ M.\ Coxeter, 
My graph, 
\emph{Proc.\ London Math.\ Soc.\ (3)} {\bf 46} (1983), 117--136.

\bibitem{CG}
H.\ S.\ M.\ Coxeter and S.\ L.\ Greitzer,
\emph{Geometry Revisited},
The Mathematical Association of America, 
Washington, D.C., 1967.

\bibitem{Cuntz}
M. \ Cuntz, 
A greedy algorithm to compute arrangements of lines in the projective plane. 
\emph{Discrete Comput. Geom.} {\bf 68(1)} (2022), 107--124.

\bibitem{Ger}
F.\ Gerbaldi,
Sul gruppi di sei coniche in involuzione,
\emph{Atti Accad.\ Sci.\ Torino}, {\bf 17} (1882), 566--580.

\bibitem{GG19}
G.\ G\'evay,
Resolvable configurations,
\emph{Discrete Appl.\ Math.} {\bf 266} (2019), 319--330.

\bibitem{GP}
G.\ G\'evay and T.\ Pisanski,
Kronecker covers, $V$-construction, unit-distance graphs and isometric
point-circle configurations,
\emph{Ars Math.\ Contemp.}, {\bf 7} (2014), 317--336.

\bibitem{GBKP}
G.\ G\'evay, N.\ Ba\v si\'c, J.\ Kovi\v c and T.\ Pisanski,
Point-ellipse configurations and related topics,
\emph{Beitr.\ Algebra Geom.} {\bf 63} (2022), 459--475.

\bibitem{Gru}
B.\ Gr{\"u}nbaum,
\emph{Configurations of {P}oints and {L}ines},
American Mathematical Society, Providence, Rhode Island, 2009.

\bibitem{GR}
B.\ Gr\"unbaum and J.\ F.\  Rigby, 
The real configuration $(21_4)$, 
\emph{J.\ London Math.\ Soc.} {\bf 41} (1990), 336--346.

\bibitem{Steen}
R. \ H. \ Jeurissen, C. \ H. \ van Os, J.\ H. \ M. \ Steenbrink,
The configuration of bitangents of the Klein curve,
\emph{Discrete Math.} {\bf 132(1-3)} (1994), 83--96.

\bibitem{Kle} 
F.\ Klein, Ueber die Transformationen siebenter Ordnung der elliptischen
Funktionen, 
\emph{Math.\ Ann.} {\bf 14} (3) (1878), 428--471.

\bibitem{McB} 
A.\ M.\ Macbeath, 
Hurwitz groups and surfaces,
in: S.\ Levy (ed.)
\emph{The Eightfold Way: The Beauty of Klein's Quartic Curve}, MSRI Publicatons {\bf 35},
Cambridge University Press, Cambridge, 1999.

\bibitem{PS}
T.\ Pisanski and B.\ Servatius,
\emph{Configurations from a Graphical Viewpoint},
Birkh\"auser Advanced Texts, Birkh\"auser, New York, 2013.

\bibitem{PR}
P.\ Pokora and J.\ Ro\'e,
The 21 reducible polars of Klein?s quartic,
\emph{Exp.\ Math.} {\bf 30} (2021), 1--18.

\bibitem{Roulleau}
X.\ Roulleau, 
Conic configurations via dual of quartic curves, 
\emph{Rocky Mountain J.\ Math.} {\bf 51} (2) (2021), 721--732.

\bibitem{Roulleau1}
X.\ Roulleau, Conic configurations via dual of quartic curves -
ancillary file with the Magma computations,
\url{https://arxiv.org/src/2002.05681v2/anc/AllMagmaComputations.txt}.

\bibitem{Z41}
M.\ Zacharias, 
Untersuchungen \"uber ebene Konfigurationen $(12_4, 16_3)$, 
\emph {Deutsche Math.} {\bf 6} (1941) 147--170.

\end{thebibliography}
\end{document}